\documentclass{lmcs} 
\pdfoutput=1
\usepackage[utf8]{inputenc}

\usepackage{lastpage}
\lmcsdoi{18}{4}{5}
\lmcsheading{}{\pageref{LastPage}}{}{}%
{Apr.~01,~2021}{Nov.~14,~2022}{}


\usepackage{hyperref}
\usepackage[all]{xy}
\usepackage{latexsym,amssymb,amsfonts,amsthm,amsmath}
\newcommand{\mlu}{\mbox{\bf MLtt$_1$}}

\newcommand{\I}{\mathsf{I}}
\newcommand{\CoI}{\mathsf{Co}\mbox{-}\mathsf{I}}
\newcommand{\step}[3]{\frac{#1}{#2}\; #3}

\newcommand{\CZF}{\mathbf{CZF}}
\newcommand{\wINAC}{\mathbf{wINACC}}
\newcommand{\wINACC}{\mathbf{wINACC}}
\newcommand{\MRboldsymbol}{\mathbf}
\newcommand{\MRmathbb}{\mathbb}

\newcommand{\MRMLS}{{\MRmathbb{MLS}}^*}

\newcommand{\MRvoll}[2]{{\mathbf{mv}}(\phantom{}^{#1}{#2})}
\newcommand{\MRpaar}[1]{\langle #1\rangle}
\newcommand{\MRhoch}[2]{\phantom{}^{#1}{#2}}

\newcommand{\MRCZF}{\ensuremath{\mathbf{CZF}}}

\newcommand{\MRREA}{\ensuremath{\mathbf{REA}}}

\newcommand{\MRZFC}{\ensuremath{\mathbf{ZFC}}}

\newcommand{\MRinn}{\!\in\!}

\newcommand{\rft}{\widetilde{\mathsf{rf}}}
\newcommand{\trt}{\widetilde{\mathsf{tr}}}
\newcommand{\sigmat}{\widetilde{\Sigma}}
\newcommand{\pit}{\widetilde{\Pi}}
\newcommand{\pp}{\mathfrak p}
\newcommand{\mfcind}{\mbox{\bf MF$_{cind}$}}

\newcommand{\mlsi}{\mbox{$\mathbb{MLS}_{ind}$}}

\newcommand{\mluif}{\mbox{\bf MLtt$_{ind}$}}
\newcommand{\mlucf}{\mbox{\bf MLtt$_{cind}$}}

\newcommand{\mtt}{\mbox{{\bf mTT}}}
\newcommand{\mttind}{{\bf mTT}$_{ind}$}
\newcommand{\mttcind}{{\bf mTT$_{cind}$}}

\newcommand{\emtt}{\mbox{{\bf emTT}}}
\newcommand{\emttind}{\mbox{{\bf emTT$_{ind}$}}}
\newcommand{\emttcind}{\mbox{{\bf emTT$_{cind}$}}}

\newcommand{\ct}{\mbox{\bf CT}}
\newcommand{\ac}{\mbox{\bf AC}}

\newcommand{\czfrea}{{\bf CZF+REA}}

\newcommand{\iduh}{$\widehat{ID}_1$}

\newcommand{\mf}{{\bf MF}}
\newcommand{\mfind}{{\bf MF$_{ind}$}}
\def\cov{\lhd} 

\begin{document}
\title[Topological Generation with Church's thesis and the Axiom of Choice]{Inductive and Coinductive Topological Generation with Church's thesis and the Axiom of Choice}
  
  \author[M.~E.~Maietti]{Maria Emilia Maietti\rsuper{a}}

\author[S.~Maschio]{Samuele Maschio\rsuper{a}}

\author[M.~Rathjen]{Michael Rathjen\rsuper{b}}

\address{Dipartimento di Matematica ``Tullio Levi Civita'',
Universit\`a di Padova, Italy}
\email{maietti@math.unipd.it, maschio@math.unipd.it}
\thanks{}	
\address{School of Mathematics,
University of Leeds, UK}
\email{M.Rathjen@leeds.ac.uk}

\thanks{Projects EU-MSCA-RISE project 731143 ``Computing with Infinite Data'' (CID), MIUR-PRIN 2010-2011 provided
support for the research presented in the paper.}

\maketitle

\begin{abstract}
In this work we consider an extension \mfcind\  of   the  Minimalist Foundation {\bf MF}  for predicative constructive mathematics 
with the addition of inductive and coinductive topological definitions sufficient to generate Sambin's Positive topologies, namely
Martin-L{\"o}f- Sambin formal topologies (used to describe pointfree formal open subsets)  equipped with a Positivity relation (used to describe pointfree formal closed subsets). In particular the intensional level of \mfcind, called \mttcind,  is  defined by extending  with coinductive topological definitions another theory  \mttind\
extending the intensional level \mtt\ of \mf\ with    the sole addition of  inductive topological definitions. In previous work we have shown that \mttind\
is consistent with adding simultaneously  formal Church's Thesis, {\bf CT}, and the Axiom of Choice, {\bf AC},  via an interpretation in
Aczel's {\bf CZF +REA}, namely  Constructive Zermelo-Fraenkel Set Theory with the Regular Extension Axiom. Our aim is to show   the expectation that the  addition of   coinductive topological definitions to \mttind\   does not  increase its  consistency strength
 by reducing  the consistency of  \mttcind{\bf +CT+AC} to the consistency of  {\bf CZF +REA} through various interpretations. We  reach our goal  in two ways.
 One way consists in first interpreting  \mttcind{\bf +CT+AC} 
in the theory extending  {\bf CZF} with the {\bf Union Regular Extension Axiom}, $\mathbf{REA}_{\bigcup}$, a strengthening of {\bf REA},  and  the Axiom of Relativized Dependent Choice, $\mathbf{RDC}$. 
The theory $\mathbf{CZF}+\mathbf{REA}_{\bigcup}+\mathbf{RDC}$ is then interpreted in $\mathbb{MLS}^\ast$, a version of Martin-L{\"o}f's type theory with Palmgren's superuniverse
$\mathbb S$.
The last step consists in interpreting $\mathbb{MLS^\ast}$ back into  {\bf CZF +REA}.
The alternative way consists in first interpreting  \mttcind{$\bf +\ac +\ct$}  directly  in a version of Martin-L{\"o}f's type theory with Palmgren's superuniverse extended with $ \ct$,
which is  then interpreted back to  {\bf CZF +REA}.
A key benefit of the  first way is that the theory $\mathbf{CZF}+\mathbf{REA}_{\bigcup}+\mathbf{RDC}$ also supports the intended set-theoretic interpretation of the extensional level of \mfcind.
Finally, all the theories considered to reach our goals,  except  \mttcind{$\bf +\ac +\ct$},  are shown to be of the same proof-theoretic strength. 
\end{abstract}

\section{Introduction}
This paper extends results from previous papers~\cite{IMMS,mmr21}  and provides a further contribution to the  project of establishing  what extensions of the  intensional level of the Minimalist Foundation in~\cite{m09}  are consistent with Formal Church's thesis and the Axiom of Choice.

The Minimalist Foundation, for short \mf,  was first ideated in~\cite{mtt} and completed in~\cite{m09}. It was introduced to serve as a constructive foundation compatible with the most relevant constructive
and classical foundations.

The reason is  that still today no foundation for constructive mathematics has affirmed itself as the standard reference  as  it happened with  Zermelo-Fraenkel axiomatic set theory for classical mathematics.
The many existing  foundations for constructive mathematics  all adopt  intuitionistic logic
for their logical reasoning   but they differ in the  accepted set-theoretic principles.  For example, their consistency status  with principles such as  the formal Church's thesis $\ct$ and the Axiom of Choice
 may be very different. Examples vary from  Heyting arithmetic in all finite types, which is constructively consistent with  both principles simultaneously (cf.~\cite{DT88}),  to foundations for Brouwer's intuitionistic mathematics in~\cite{DT88} which are inconsistent with $\ct$,  to  foundations like Aczel's Constructive Set Theory, for short {\bf CZF},  which
is constructively consistent with $\ct$ together with some restricted forms of the Axiom of Choice (see~\cite[9.2,10.1]{rathjen2006b}), whereas the combination of 
$\ct$ and the full axiom of choice together with {\bf CZF} gives rise to inconsistency.

Moreover, another peculiarity of the Minimalist Foundation \mf\  in~\cite{m09}  is that of being  a two-level foundation, as required in~\cite{mtt}. Indeed, \mf\ is
equipped with an extensional level, called \emtt, formulated in a language close to that of  everyday mathematical practice and
interpreted via a quotient model  in a further intensional level, called \mtt,  designed as a type-theoretic base for a proof-assistant.

 The notion of two-level  constructive foundation was  introduced in~\cite{mtt}, motivated by the need to founding  constructive mathematics on a theory consistent  with the axiom of choice and Church's thesis $\ct$ (via an extension of Kleene realizability interpretation of intuitionistic arithmetics) including a way to formalize usual
 desirable extensional  set-theoretic principles, such as extensional equality of functions, which are inconsistent
with the axiom of choice and $\ct$   (see also~\cite{DT88}). 

Only recently, in~\cite{IMMS},  it was shown that the intensional level \mtt\ of \mf\  is consistent with the axiom of choice $\ac$ and the formal Church's thesis $\ct$, 
expressed as follows:
 $\ac$ states
 that  from any total relation
we can extract a type-theoretic function
\[ ({\bf AC})\;(\forall x\in A)\,(\exists y\in B)\,R(x,y)\rightarrow (\exists f\in (\Pi x\in A)\,B)\,(\forall x\in A)\,R(x,\mathsf{Ap}(f,x))\]
with $A$ and $B$ generic collections and $R(x,y)$ any relation, while
$\ct$ (see also~\cite{DT88}) states that from any total relation on natural numbers
we can extract a (code of a) recursive function by using the Kleene predicate $T$ and the extracting function $U$
 \[({\bf CT})\;(\forall f\in \mathsf{N}\rightarrow\mathsf{N})(\exists e\in \mathsf{N})\,(\forall x\in \mathsf{N})\,(\exists z\in \mathsf{N})\,(T(e,x,z)\wedge \mathsf{Ap}(f,x)=_\mathsf{N}U(z)).\]
Furthermore,  the whole \mf\  turned out to be not only constructive but also predicative in the strict sense of Feferman, as first shown in~\cite{ms16}.


Another example of two-level foundation   is the theory \mfind\ proposed in~\cite{mmr21}. \mfind\ extends \mf\  with all the inductive definitions necessary to formalize  the inductive  topological methods  developed in 
~\cite{CSSV03,cms13}  within Formal Topology which is  a predicative and constructive approach to topology  put forward by P. Martin-L{\"o}f and G. Sambin 
in~\cite{S87}.  Indeed, it was expected that \mf\ is not able to formalize all such inductive methods, due to its minimality.

In particular, in~\cite{mmr21} we showed that the intensional level \mttind\ of \mfind\   is consistent with $\ac$ and $\ct$ by extending the realizability semantics  used in~\cite{IMMS,R93,RG94,antirat}.  A major improvement of the realizability semantics
 in~\cite{mmr21}  was its formalization in a constructive theory such as the (generalized) predicative  set theory  \czfrea,
 namely Aczel's constructive Zermelo-Fraenkel set theory extended with the regular extension axiom  {\bf REA} (see~\cite{aczel86}). Instead, the semantics in~\cite{IMMS} was formalized  in Feferman's predicative theory of non-iterative fixpoints \iduh\
 which is governed by classical logic.  Both semantics extend Kleene realizability semantics of intuitionistic arithmetic and hence they validate  $\ct$ and $\ac$.

This paper aims to further extend \mfind\  by  also including the coinductive topological methods  introduced  by  P. Martin-L{\"o}f and G. Sambin in~\cite{somepoints,postop}
 to describe both open and closed subsets of a point-free topology,  later developed as Positive Topology in~\cite{somepoints}.
 Instances of such  inductive and coinductive definitions  seem necessary  to provide   a constructive and predicative version  of results like
  the coreflection of the category of locales in that of open locales  in~\cite{mv04},  or the embedding of the category of locales in that of positive topologies in~\cite{emb}.

  To meet our purpose, here  we build a two-level extension  \mfcind\  of \mfind, by enriching both levels of \mfind\   with Martin-L{\"o}f and Sambin's  coinductive definitions necessary to generate Sambin's Basic Topologies in~\cite{somepoints, postop},  namely
 inductively generated basic covers (used to describe pointfree formal open subsets)  enriched with a positivity relation (used to describe pointfree formal closed subsets),

Then, we  prove two main results concerning  \mfcind. First, we show that \mfcind\ is a two-level theory in the sense of~\cite{mtt} by showing that its intensional level, called \mttcind, 
  is consistent with $\ac$ and $\ct$.  Second,
   we  show   the expectation that the  addition of   coinductive topological definitions to \mttind\    does not  increase its  consistency strength
 by reducing  the consistency of  \mttcind{\bf +CT+AC} to the consistency of  {\bf CZF +REA} through various interpretations.
We  reach this goal  in two ways.
 
 One way consists in first interpreting  \mttcind{\bf +CT+AC} 
in the theory extending  {\bf CZF} with the {\bf Union Regular Extension Axiom}, $\mathbf{REA}_{\bigcup}$, a strengthening of {\bf REA}, as well as with the Axiom of Relativized Dependent Choice, $\mathbf{RDC}$. 
The theory $\mathbf{CZF}+\mathbf{REA}_{\bigcup}+\mathbf{RDC}$ is then interpreted in $\mathbb{MLS}^\ast$, a version of Martin-L{\"o}f's type theory with Palmgren's superuniverse
$\mathbb S$.
The last step consists in interpreting $\mathbb{MLS^\ast}$ back into  {\bf CZF +REA}.

The alternative way consists in first interpreting  \mttcind{$\bf +\ac +\ct$}  directly  in a version of Martin-L{\"o}f's type theory with Palmgren's superuniverse extended with $ \ct$,
which is then interpreted  in  {\bf CZF +REA} by extending the realizability semantics  in~\cite{mmr21}.

A  key benefit of the first way   is that the intermediate theory $\mathbf{CZF}+\mathbf{REA}_{\bigcup}+\mathbf{RDC}$ also supports the  intended set-theoretic interpretation of the extensional level of \mfcind. 


Finally, all the theories considered to reach our goals, except  \mttcind{$\bf +\ac +\ct$},  are shown to be of the same proof-theoretic strength.   

We leave it to future work to establish    the consistency strength  of  \mttind\ and \mttcind,
given that it is still an open problem to establish that of \mtt\ itself.                    

Another future goal would be to apply   the  realizability interpretations presented here  to build  predicative variants of Hyland Effective Topos as in~\cite{misam21} but in a constructive meta-theory  such as  {\bf CZF +REA}.

Last but not least there is the obvious question of  generalizing our result by extending  \mf\   with more general   coinductive definitions  
such as those of~\cite{LIN89,antirat} including streams and systems, or  strictly positive coinductive types  in~\cite{coinMLTT} or other coinductive definitions like those applied in~\cite{VDJSL,curicind, BT21}.
 The theory  $\mathbf{CZF}+\mathbf{REA}_{\bigcup}+\mathbf{RDC}$ is very capacious in that it allows for the interpretation of all of these types as sets. It is the ideal
  axiomatic theory for a very general treatment of inductive and coinductive definitions in its class and set forms, delineating the necessary conditions for such definitions to give rise to sets rather than classes (see~\cite{czf2} chapters 12 and 13). But it is not clear how to turn such  extensions into  proper two-level foundations according to~\cite{mtt}, by identifying an intensional type theoretic version of such coinductive definitions which
 is preserved by the syntactic quotient completion in~\cite{m09}. Actually  it is not clear how to define  such generic coinductive  types  within intensional type theory with well-behaved rules (see for example~\cite{failure}). Moreover, the current encodings of coinductive types in~\cite{coinMLTT,coinHott}
are performed by employing extensional principles like extensional equality of functions which is not validated in our realizability interpretation by its inconsistency
with $ \ac+\ct$. 

\section{The extension \mfcind\  with  inductive and coinductive topological definitions.}\label{mfcind}
Here we  introduce an extension of the Minimalist Foundation \mf,  called \mfcind,
 equipped with  the inductive and coinductive topological  definitions necessary to generate Sambin's Basic Topologies  and Positive Topologies in~\cite{somepoints}
 within the approach of  Martin-L{\"o}f -Sambin's Formal Topology in~\cite{S87}.
 In particular, the notion of Basic Topology constitutes  an enrichment of  the notion of  \emph{basic cover} in~\cite{batsampre}  
 (used to describe open subsets of a point-free topology)  with a  \emph{positivity relation} used to describe closed  subsets in a primitive way.

%
 
%

 
 In more detail,  we recall that the   notion of  {\bf basic cover}   in~\cite{batsampre}  aims to represent   complete suplattices in a predicative and constructive way
by an infinitary relation  indicated with the notation
\[a\cov V\]
between elements $a$ of a set $A$, thought of as \emph{basic opens}, and subsets $V$ of $A$, meaning that \emph{the basic open $a$ is covered by
    the union of basic opens in the subset $V$}.

The elements of the complete suplattices represented by the basic cover,  and called \emph{formal open subsets},  are  the fixpoints of the \emph{closure operator}
\[\cov(-) : \mathcal{ P}(A)\ \longrightarrow \  \mathcal{ P}(A)\]
defined by putting $\cov(V)\, \equiv\, \{ \ x\in A\ \mid\ x\cov
V\ \}$.
These suplattices are \emph{complete} with respect to families of subsets \emph{indexed over a set}.

Then  any  basic cover $a\cov V$ gives rise to a  \emph{Basic Topology} as in~\cite{somepoints}, 
 if it is paired with a {\bf positivity relation }
\[a\ltimes V\]
between elements $a$ of  a set $A$ and subsets $V$ of $A$, meaning that \emph{there is a point in  the basic open $a$ whose
basic neighbourhoods are all in $V$} (see~\cite{somepoints}). In more detail, the positivity relation is required 
 to satisfy the following   \emph{compatibility condition} with the basic cover
\begin{center}
    \begin{tabular}{l}
$
\displaystyle{\frac{  a\ltimes V   \qquad  a\cov U }
{\displaystyle   \exists_{x\in A}\  (\  x\ltimes V \  \ \& \  x\,\epsilon\, U\ )}}$
\qquad
\end{tabular}
\end{center}
and to induce an  \emph{interior operator}
   \[\ltimes(-) : \mathcal{ P}(A)\ \longrightarrow \  \mathcal{ P}(A)\]
defined by putting
$\ltimes(V)\, \equiv\, \{ \ x\in A\ \mid\ x\ltimes
V\ \}$.
Then, the  fixpoints of this operator  form the suplattice of formal closed subsets.

A basic topology whose  basic cover satisfies a convergence property is renamed as \emph{positive topology} in~\cite{postop}~\footnote{Positive topologies were called ``balanced formal topologies'' in~\cite{somepoints}.}, because  the resulting complete suplattice of $\cov$-fixpoints actually forms
a \emph{locale} enriched with a suplattice of formal closed subsets.  A positive topology is also a  formal topology in the sense of~\cite{S87} if the positivity relation $ a\ltimes A$ defines a positivity predicate $\mathsf{Pos}(a)$
meaning that  \emph{the basic open $a$ is inhabited by a point}. Therefore the notion of positive topology appears to be a strengthening of the notion of formal topology in~\cite{S87}
with the modification of just requiring the suplattice of formal opens to form a  locale rather than an open locale as  in~\cite{S87}.

Classically,  a positivity relation can be associated to any basic cover in the form 
\[a\ltimes V \ \equiv\  \neg a \cov \neg V\]
but constructively one needs to add a primitive operator (see~\cite{postop}).

A powerful method to generate basic and positive topologies  is to pair the inductive generation of basic covers  
 introduced in~\cite{CSSV03} with the coinductive generation of the positivity relation
 introduced by Martin-L{\"o}f and Sambin  in~\cite{somepoints}.  
%
Since  \mf\  cannot formalize
 all the inductively generated basic covers, like for example the inductive topology of the Baire space (see~\cite{mmr21}),
 it is clear that  to formalize such inductive and coinductive topological methods
we need to build an extension of  \mf\    like  \mfcind\  that  we are going to describe now.

 By  extension of \mf\  we actually mean \emph{a two-level extension}, namely we extend both the intensional level \mtt\ of \mf\
and its extensional level \emtt\ to form a two-level theory satisfying the requirements of a constructive foundation  in~\cite{mtt}.

In~\cite{mmr21} we  built  a  two-level theory,  called \mfind, 
by extending both levels of  \mf\  with the inductive definitions introduced in~\cite{CSSV03} necessary to generate basic covers  in~\cite{batsampre,cms13}.
 
 Knowing that the  generation of basic topologies consists in enriching the inductive generation of basic covers  in~\cite{CSSV03}  with the coinductive generation of positivity relations,
we simply  extend  both levels of \mfind\ with Martin-L{\"o}f-Sambin's  coinductive  positivity relations~\cite{somepoints}.
  
 Now, we define in detail  the two-level structure of \mfcind\  by first  describing its extensional level \emttcind\  
    and then its   intensional level \mttcind.

\subsection{The extensional level \emttcind\ of \mfcind}
The extensional level \emttcind\ of \mfcind\ is defined as the extension of the extensional level \emttind\  of \mfind\ in~\cite{mmr21} with the rules   generating  
 positivity relations by coinduction.

In \emttcind\ we still keep the four kinds of types of \emtt, namely
{\bf collections}, {\bf sets},  {\bf propositions} and {\bf small propositions}
according to the following subtyping relations:
\[
\xymatrix@C=4em@R=1em{
 {\mbox{\bf small propositions}}\ar@{^{(}->}[dd]\ar@{^{(}->}[rr]&  &
 {\mbox{\bf sets}}\ar@{^{(}->}[dd] \\ 
&&\\
  {\mbox{\bf propositions}}\ar@{^{(}->}[rr] && {\mbox{\bf collections}}
}
\]
where collections  include the power-collection
${\mathcal P}(A)$ (which is not a set!)  of  any set $A$ and small propositions are defined as those propositions closed
under intuitionistic connectives, and quantifiers and equalities restricted to sets.

In addition to the rules of  \emtt,  in \emttind\ and hence in \emttcind,  we have new primitive \emph{small propositions}
\[a\triangleleft_{I,C} V\ prop_s\]
 expressing that \emph{the basic open $a$ is covered by the union of basic opens in $V$}
 for any $a$ element of a \emph{set} $A$, $V$ subset of $A$,
 assuming that \emph{the basic  cover is
generated by a family
of (open) subsets  of $A$
indexed on  a family of  sets $I(x)\ set \ [x\in A]$ and represented by}
  \[C(x,j) \in {\mathcal P}( A)\ [x\in A, j\in I(x)] \]
which we  call \emph{axiom-set} and indicate with the abbreviation $I,C$.

As for \emtt\   and  \emttind,  in \emttcind\ we can can define a \emph{subset membership} relation
\[a\,\epsilon\, V\]
between a subset $V\in  {\mathcal P}( A)$ and an element $a$  of a set $A$ as in~\cite{m09},
which is different from the \emph{primitive} typing membership $a\in A$  which is  a judgement.

    We also adopt
    the convention of writing $\phi\ true$ for a proposition $\phi$
    instead of $\mathsf{true}\in \phi$ as in~\cite{m09}.

Recall that  the rules of \emttind\  in~\cite{mmr21} were obtained by extending those of \emtt\  with the addition of the rules
defining  the basic cover $a\triangleleft_{I,C} V\ prop_s$ inductively generated from an axiom-set  $I,C$.

The rules of \emttcind\ are obtained by extending those of \emttind\
with the following  ones defining  positivity relations by coinduction:
\\
    
  {\bf Rules of coinductive Positivity relations in \emttcind}
\[\begin{array}{l}
\\
\mbox{\rm F- Pos} \
\displaystyle{\frac{\begin{array}{l}
      A \ set \ \ \ \  I(x)\  set \  [x\in A] \ \ \ C(x,j) \in {\mathcal P}( A)\ \ [x\in A, j\in I(x)]\\
      V\in {\mathcal P}( A)\qquad a\in A\ \end{array} }
  {\displaystyle  a \ltimes_{I,C}V \ \ prop_s}}
\\[20pt]
\mbox{\rm crf- Pos} 
\displaystyle{\frac{\begin{array}{l}
      A \ set \ \ \ \  I(x)\  set \  [x\in A] \ \ \ C(x,j)\in {\mathcal P}( A) \ \ [x\in A, j\in I(x)]\\
   V\in {\mathcal P}( A) \qquad a\in A\
   \qquad\qquad    a \ltimes_{I,C}V   \ true \end{array} }
  {\displaystyle a\,\epsilon\, V\ true\ }}\\
\end{array}
  \]
\[\begin{array}{l}
\mbox{ax-mon-Pos} 
\displaystyle{
  \frac{\begin{array}{l} A\ set \ \ \ \  I(x)\  set \  [x\in A] \ \ \ C(x,j)\in {\mathcal P}( A)  \ \ [x\in A, j\in I(x)]\\
      a\in A\qquad i\in I(a)\qquad  \qquad  V\in {\mathcal P}( A)\qquad  a \ltimes_{I,C}V  \ true \end{array} }
  {\displaystyle  \exists_{y\in A} \ ( y\,\epsilon\, C(a,i)\ \& \  y \ltimes_{I,C}V \ )\ true}}
\\[20pt]
\mbox{cind-Pos} 
\quad
\displaystyle{
  \frac{\begin{array}{l} A\ set \ \ \ \  I(x)\  set \  [x\in A] \ \ \ C(x,j)\in {\mathcal P}( A)  \ \ [x\in A, j\in I(x)]\\
      a\in A  \qquad  V\in {\mathcal P}( A)\\
  P(x)\  prop\ [x\in A]\qquad 
        \mathsf{split}(V,P)\ true \qquad  P(a)\ true \end{array} }
 {a\ltimes_{I,C}V \ true\ }}
\end{array}\]
    where
    \[\begin{array}{rl}
      \mathsf{split}(V,P)\ \equiv\ & \forall_{x\in A}\ (\  P(x)\ \rightarrow\ ( \ x\varepsilon V\ 
    \ \&\  \forall_{z \in I(x)}\ \exists_{y \in A} \ (\ y\,\epsilon\, C(x,z) \ \&  \ P(y) \ ) \ ) \ )\end{array}\]

The coinductive generation of positivity relations is crucial to provide a predicative and constructive representation of results
concerning applications of  the theory of  positivity relations and basic topologies  to  that of locales, namely the definition of closed  sublocales in~\cite{subloc},
the relation between formal topology and   Brouwer's  principles  in~\cite{cont,postop} and the connections between locales, open locales and positive topologies in~\cite{mv04,emb}.
More in detail,   in \emttcind\   we can formalize   the embedding of the category of   locales represented as inductively generated formal covers  into  the category of positive topologies  in~\cite{emb},  by embedding any  inductive formal cover  $\cov_{I,C}$
as the positive topology with the same basic cover  enriched with the coinductive  positivity relation $\ltimes_{I,C} $.
Moreover, 
 in \emttcind\  we can formalize the coreflection in~\cite{mv04} of the category of locales represented as  inductively generated  formal covers 
 into its subcategory of open locales represented as inductively generated formal topologies
by coreflecting  any convergent basic cover   $\cov_{I,C}$ inductively generated  from an axiom-set $ I,C$,
into the  inductively generated formal topology   $(A, \cov_{I^+,C^+}, \mathsf{Pos}) $  whose
 positivity predicate $\mathsf{Pos}$ is defined in terms of the coinductive  positivity relation $\ltimes_{I,C} $ as
\[\mathsf{Pos}(a)\,  \equiv\  a \ltimes_{I,C} A\]
and its basic cover $\cov_{I^+,C^+}$  is inductively generated by a new axiom-set  $ I^+(x)\  set \  [x\in A]$ and $C^+(x,j)\in {\mathcal P}( A)\  [x\in A,j\in I^+(x)]$ extending the given  one  by adding  for each basic open $a$ an index $k$ such that
\[C^+(a,k)\, \equiv\, \{\  x\in A \  \mid\  \mathsf{Pos}(a)\ \}.\]

   In view of interpreting   \emttcind\  in the intensional level  \mttcind, 
as noted  in~\cite{mmr21} for basic covers,    it is crucial  that basic topologies  defined by  $\cov_{I,C}$ and $\ltimes_{I,C} $  on a  quotient set base $B/R$
 can be equivalently presented by   a cover and a positivity relation  on the set $B$ itself
which behaves like  $\cov_{I,C}$ and $\ltimes_{I,C} $, respectively,  but in addition 
it considers as equal opens those elements which are related by $R$.

In order to properly show this fact, 
we recall  a correspondence between subsets of $B/R$
    and subsets of $B$  in~\cite{mmr21}:
     \begin{defi}
      In \emttcind, 
      given  a quotient set $B/R$, for any subset $W\in \mathcal{P}(B/R)$
      we define
       \[\mathsf{es}(W)\, \equiv \, \{\ b\in B\ \mid\  [b]\,\epsilon\, W\     \}\]
        and given any $V\in  \mathcal{P}(B)$ we define $\mathsf{es}^{-}(V)\, \equiv \, \{\ z\in B/R\ \mid\  \exists_{b\in B}\ (\ b\,\epsilon\, V\ \wedge\ z=_{B/R}[b] \ )\}$.
    \end{defi}
  
    \begin{defi}\label{indquo}
Given an axiom set  represented by a set $A\,\equiv\, B/R$ with
$ I(x)\  set\  [x\in A]$ and $  C(x,j)\in {\mathcal P}(A) \ \ [x\in A, j\in I(x)]$, we define a new axiom set as follows:
      \[\begin{array}{l}
 A^R\, \equiv\, B\qquad \qquad I^R(x)\, \equiv\, I([x]) + \ (\Sigma y\in B)\ R(x,y) \qquad \qquad \mbox{ for } x\in B\\
\end{array}
      \]
where $C^R(b,j)$ is the formalization of
      \[C^R(b,j)\, \equiv\, \begin{cases}
  \mathsf{es}(\, C([b],j)\, ) &  \mbox{ if } j\in  I([b])\\
  \{\, \pi_1(j) \, \} &  \mbox{ if } j\in (\Sigma y\in B)\ R(b,y) \\
      \end{cases}\]
for $b\in B$ and $j\in I^R(x)$.

\noindent We then call $\cov_{I,C}^R$ and $\ltimes_{I,C}^R$  the inductive basic cover  and the coinductive positivity relation
generated from this axiom set, respectively.
      \end{defi}

  From~\cite{mmr21} we know that:
    \begin{lem}\label{eqcov}
      For any axiom set in \emttind\ represented by a set $A\,\equiv\, B/R$ with
      $ I(x)\  set\  [x\in A]$ and $  C(x,j)\in {\mathcal P}(A) \ \ [x\in A, j\in I(x)]$, the suplattice of fixpoints defined by $\cov_{I,C}$ is isomorphic
      to that defined by $\cov_{I,C}^R$ by means of an isomorphism of suplattices.
\end{lem}
    Analogously we can show that
    \begin{lem}\label{ceqcov}
      For any axiom set in \emttind\ represented by a set $A\,\equiv\, B/R$ with
      $ I(x)\  set\  [x\in A]$ and $  C(x,j)\in {\mathcal P}(A) \ \ [x\in A, j\in I(x)]$, the suplattice of fixpoints  defined by $\ltimes_{I,C}$ is isomorphic
      to that defined by $\ltimes_{I,C}^R$ by means of an isomorphism of suplattices.
\end{lem}
    \begin{proof}
        Just observe that for any subset  $W$ of $B/R$
        which is a fixpoint for $\ltimes_{I,C}$ the subset
        $\mathsf{es}(W)$ is a fixpoint for $\ltimes_{I,C}^R$ by cind-Pos.
        Conversely,
        for any subset  $V$ of $B$
        which is a fixpoint for $\cov_{I,C}^R$ the subset 
        $\mathsf{es}^{-}(V)$ is a fixpoint for $\ltimes_{I,C}$ always by cind-Pos.
        \end{proof}

\subsection{The intensional level \mttcind}
Here we define the intensional level  \mttcind\ of \mfcind\ 
as an extension of the intensional level \mtt\ of \mf\ 
capable of interpreting the extensional level \emttcind.
We actually describe  \mttcind\  as  an extension of \mttind\  in~\cite{mmr21} with the rules generating a positivity relation by coinduction, as done
for \emttcind\  with respect to  \emttind.

As in \mtt\ in~\cite{m09} and in \mttind,  in \mttcind\ we have the same four kinds
of types as in \emttcind\  with the difference that  in \mttcind\ power-collections of sets
are replaced by a  \emph{collection of small propositions} $\mathsf{prop_s}$
and function collections $A\rightarrow \mathsf{prop_s}$ for any set $A$.
Such collections are enough to interpret power-collections of sets in \emttcind\  within a quotient model of dependent extensional types
built over \mttcind, as shown 
in~\cite{m09} when interpreting  \emtt\ in \mtt.

In addition to small proposition constructors of  \mtt,   in \mttcind\  and hence  in \mttind\  we have
new small propositions $ a\triangleleft_{I,C} V\ prop_s$
with corresponding   new proof-term constructors associated to them
in order to represent a proof-relevant version of 
  inductively generated basic covers
 so that
judgements asserting that some proposition is true  in \emttcind\ 
are turned into   judgements of \mttcind\
producing a proof-term of the corresponding proposition.

Moreover,  in \mttcind\ as in \mtt\  and in \mttind,  the universe of small propositions  is defined  in the version  \`a la Russell~\footnote{A version  of \mtt\ with the universe of small propositions  \`a la Tarski can be found  in~\cite{ms16}.}. 


When expressing the rules of inductive basic covers
 we used the abbreviation \[ a\,\epsilon\, V \qquad \mbox{ to mean }\qquad \mathsf{Ap}(\, V\, , \, a\, )\]
for any set $A$,  any small propositional function $ V\in A\rightarrow\mathsf{prop_s}$ and any element $a\in A$.
\\

   Formally, \mttcind\ is obtained by extending the rules of \mttind\ in~\cite{mmr21}
with the following rules defining  positivity relations by coinduction in the form of axioms with no equality rules:
    \\

  {\bf Axioms of  coinductive Positivity relations in \mttcind}
  \\
  \[\begin{array}{l}
\mbox{\rm F-Pos} \
\displaystyle{\frac{\begin{array}{l}
      A \ set \ \ \ \  I(x)\  set \  [x\in A] \ \ \ C(x,j) \in A\rightarrow 
      \mathsf{prop_s}\ [x\in A, j\in I(x)]\\
        V\in A\rightarrow\mathsf{prop_s}\qquad a\in A \end{array} }
  {\displaystyle a\ltimes_{I,C}V\ prop_s }}\\[20pt]
\mbox{\rm crf-Pos} \
\displaystyle{\frac{\begin{array}{l}
      A \ set \ \ \ \  I(x)\  set \  [x\in A] \ \ \ C(x,j) \in A \rightarrow \mathsf{prop_s} \ \ [x\in A, j\in I(x)]\\
        V\in A \rightarrow \mathsf{prop_s}
   \qquad a\in A\qquad q\in a\ltimes_{I,C}V  \ \end{array} }
  {\displaystyle \mathsf{ax_1}(a,q)\in  a\,\epsilon\, V}}
\\[20pt]
  \end{array}\]
  \[\begin{array}{l}

    \mbox{\rm ax-mon-Pos} \
\displaystyle{
  \frac{\begin{array}{l} A\ set \ \ \ \  I(x)\  set \  [x\in A] \ \ \ C(x,j) \in  A \rightarrow \mathsf{prop_s}\ [x\in A, j\in I(x)]\\
        V\in  A \rightarrow \mathsf{prop_s} \qquad\ \
      a\in A\qquad i\in I(a)\\
       q\in a\ltimes_{I,C}V \end{array} }
  {\displaystyle \mathsf{ax_2}(a,i,q)\in \exists_{y\in A} \ (\ y\epsilon C(a,i)\ \&\ y\ltimes_{I,C}V\ )  \  }}
\\[20pt]
\mbox{cind-Pos} 
\displaystyle{ \frac{ 
    \begin{array}{l}    A\ set \ \ \ \  I(x)\  set \  [x\in A] \ \ \
      C(x,j)\  \in  A\ \rightarrow \mathsf{prop_s} \ [x\in A, j\in I(x)]\\[5pt]
      P(x)\  prop\ [x\in A]\qquad
      V\in A\ \rightarrow\  \mathsf{prop_s}\qquad \\[5pt]
      a\in A\qquad m\in P(a)\\[5pt]
      q_1(x,z)\in V(x)\ [ x\in A, z\in P(x)]\qquad \qquad \\[5pt]
      q_2(x,j, z) \in \exists_{y\in A}  ( P(y)\ \& \ y\,\epsilon\, C(x,j) )\ 
      [ x\in A, j\in I(x),  z\in P(x)]    \end{array} }{\mathsf{ax_3}(a,m, q_1,q_2) \in  a\ltimes_{I,C}V}}
  \\[20pt]
  \end{array}\]

  As noted in~\cite{mmr21}  for  the cover relation, the positivity relation preserves extensional equality of subsets represented as small propositional functions:
  
\begin{lem}\label{prepos}
      For any axiom set in \mttcind\ represented by a set $A$ with
      $ I(x)\  set\  [x\in A]$ and $  C(x,j)\in A\ \rightarrow \mathsf{prop_s} \ \ [x\in A, j\in I(x)]$, for any propositional functions
      $  V_1\in  A\ \rightarrow \mathsf{prop_s}$ and $  V_2\in  A\ \rightarrow \mathsf{prop_s}$,  we can derive for a proof-term $q$ 
      \[q\in V_1=_{ext} V_2\ \rightarrow\  \mathsf{Pos}_{I,C} (a, V_1)=_{ext} \mathsf{Pos}_{I,C} (a, V_2)\]
      where for any small propositional functions $W_1$ and $W_2$
on a set $A$ we use the following abbreviation $W_1=_{ext}W_2 \, \equiv\, \forall_{x\in A}\ (\ W_1(x)\ \leftrightarrow\ W_2(x)\ )$.
    \end{lem}
    
        Recall that the interpretation of \emtt\ in \mtt\ in~\cite{m09}, as well as that of \emttind\ in \mttind\ in~\cite{mmr21}, 
        interprets a set $A$ as an \emph{extensional quotient} defined in \mtt\
        as a set $A^J$  of \mtt, called \emph{support}, equipped
        with an equivalence relation $=_{A^J}$ over $A^J$, as well as families of
        sets are interpreted as families of extensional sets preserving the equivalence
        relations in their telescopic contexts. Now, 
         Lemma~\ref{prepos} suggests that  we can interpret a coinductive positivity relation on a set $A$ of \emttcind\  
         within \mttcind\  as a coinductive positivity relation of \mttcind\ on  the support $A^J$ by enriching the interpretation of the axiom-set in \mttcind\ with the equivalence
         relation $=_{A^J}$ in a similar way to Definition~\ref{indquo} as follows:
         \begin{defi}
          For any axiom set in \mttcind\   represented by a set $A$ with
      $ I(x)\  set\  [x\in A]$ and $  C(x,j)\in A\ \rightarrow \mathsf{prop_s} \ \ [x\in A, j\in I(x)]$ 
      and for any given equivalence relation
      $x=_Ay \in   \mathsf{prop_s}\ [x\in A,y\in A]$ turning $A$ into an extensional set 
      as well as  the family of set  $ I(x)\  set\  [x\in A]$ and propositional functions $  C(x,j)\in A\ \rightarrow \mathsf{prop_s} \ \ [x\in A, j\in I(x)]$
      into an extensional family of sets and extensional propositional functions preserving $=_{A}$ according to the definitions in~\cite{m09},
         we define a new axiom set as follows
           \[\begin{array}{l}
 A^{=_A}\, \equiv\, A\qquad \qquad I^{=_A}(x)\, \equiv\, I(x) + (\Sigma y\in A)(  x=_A y)\qquad \qquad \mbox{ for } x\in A\\
\end{array}
           \]
where $C^{=_A}(a,j)$ is the formalization of
           \[C^{=_A}(a,j)\, \equiv\, \begin{cases}
   C(a,j)\, &  \mbox{ if } j\in  I(a)\\
  \{\, \pi_1(j) \, \} &  \mbox{ if } j\in (\Sigma y\in A)( a=_A y)\\
           \end{cases}\]
for $a\in A$ and $j\in I^{=_A}(x)$.

\noindent We then call $\cov_{I,C}^{=_A}$ and $\ltimes_{I,C}^{=_A}$ respectively the inductive basic cover and the coinductive positivity relation
generated from this axiom set.
      \end{defi}

We  are now ready to interpret \emttcind\ in the quotient model over \mttcind\ built as in~\cite{m09} by extending the interpretation of
\emttind\ in \mttind\ in~\cite{mmr21} as follows:
    \begin{prop}\label{fullcv}
      The interpretation of \emtt\ in \mtt\ in~\cite{m09} extends to an interpretation of \emttcind\ in \mttcind\ by interpreting an
       inductive basic cover $a \cov_{I,C} V$  and a coinductive positivity relation $a \ltimes_{I,C} V$   for $a\in A$
      and $V\in {\mathcal P}(A)$
    respectively  as the inductive basic cover $\cov_{I^J,C^J}^{=_{A^J}}$ and the coinductive positivity relation $\ltimes_{I^J,C^J}^{=_{A^J}}$
    over the support $A^J$ of the interpretation of $A$.
    \end{prop}
       \begin{proof}
       The given interpretation coincides with the interpretation of \emttind\  in \mttind\ given in~\cite{mmr21} for what regards
   inductive basic covers.
      To check its correctness for what regards coinductive positivity relations, just observe that  $\ltimes_{I^J,C^J}^{=_{A^J}}$ 
       is an extensional proposition over the extensional set interpreting $A$ and over
      the interpretation of $ {\mathcal P}(A)$  in the sense of~\cite{m09}.
       \end{proof}

 \section{The Martin-L{\"o}f's  type theories \mlucf\ and \mlsi\ }  
 In order to derive the consistency of  \mttcind\   with $\ac+\ct$  from the consistency of $\mathbf{CZF}+\mathbf{REA}$ we introduce two auxiliary theories \mlucf\ and \mlsi\ 
  with the idea
 of reducing the consistency of    \mttcind\   with $\ac+\ct$ 
 to the consistency of \mlucf\ or \mlsi\ with just $\ct$. The crucial point is that both \mlucf\ and \mlsi\ are intensional Martin-L{\"o}f's type theories 
 governed by  Curry-Howard's interpretation of  propositions-as-sets  which validates the axiom of choice $\ac$ internally. This idea was already implemented  to obtain similar results both in~\cite{IMMS}  and in~\cite{mmr21}.
 
The theory \mlucf\ is an extension of the theory \mluif\ in~\cite{mmr21} with  rules generating coinductive positivity relations, which model those of \mttcind.
 
Instead, the theory \mlsi\   includes  Palmgren's superuniverse $\mathcal{S}$ in~\cite{suppal}  which is used to interpret  small propositions
of \mttcind.
By contrast,   the fragments of  Martin-L{\"o}f's type theory  \mlucf\ here,  \mlu\ in~\cite{IMMS}  and \mluif\ in~\cite{mmr21}   just contain a single universe (whilst closed
under some primitive inductive definitions in \mluif\ and furthermore under some coinductive definitions in \mlucf) to interpret small propositions of the considered extensions of \mtt.  Moreover, as done for the universe in the theory \mluif\  in~\cite{mmr21}, we close  the superuniverse under inductive covers rather than
well-founded sets as  in~\cite{antirat}.  
The presence of the superuniverse in  \mlsi\   is crucial in defining the interpretation of the coinductive positivity relations of \mttcind\  thanks to the closure of the superuniverse under universe operators  and the axiom of choice.

 \subsection{The theory \mlucf\ } 
 \label{mlcpos}
 The theory $\mlucf$ is obtained by extending  \mluif\ in~\cite{mmr21} (which has one universe $\mathsf{U}_0$ \`a la Tarski) with the following axioms for coinductive positivity relations:\\
{\bf Coinductive positivity relations in $\mlucf$}
  
  \[\begin{array}{l}
\mbox{\rm F-Pos} \
\displaystyle{\frac{\begin{array}{l}
      s\in \mathsf{U}_0 \ \ \ \  i(x)\in \mathsf{U}_0\  [x\in \mathsf{T}(s)] \ \ \ c(x,j)\in \mathsf{T}(s)\rightarrow  \mathsf{U}_0\ [x\in \mathsf{T}(s), j\in \mathsf{T}(i(x))]\\
        v\in \mathsf{T}(s)\rightarrow\mathsf{U}_{0}\qquad a\in\mathsf{T}(s)  \end{array} }
  {\displaystyle a\widehat{\ltimes}_{s,i,c}v\in \mathsf{U}_0 }}\\[20pt]
\mbox{\rm crf-Pos} \
\displaystyle{\frac{\begin{array}{l}
     s\in \mathsf{U}_0 \ \ \ \  i(x)\in \mathsf{U}_0\  [x\in \mathsf{T}(s)] \ \ \ c(x,j)\in \mathsf{T}(s)\rightarrow  \mathsf{U}_0\ [x\in \mathsf{T}(s), j\in \mathsf{T}(i(x))]\\
       v\in \mathsf{T}(s)\rightarrow\mathsf{U}_{0}\qquad a\in\mathsf{T}(s)\qquad q\in a\ltimes_{s,i,c} v   \ \end{array} }
  {\displaystyle \mathsf{ax_1}(a,q)\in  a\,\epsilon\, v}}
\\[20pt]

    \mbox{\rm ax-mon-Pos} \
\displaystyle{
  \frac{\begin{array}{l}   s\in \mathsf{U}_0 \ \ \ \  i(x)\in \mathsf{U}_0\  [x\in \mathsf{T}(s)] \ \ \ c(x,j)\in \mathsf{T}(s)\rightarrow  \mathsf{U}_0\ [x\in \mathsf{T}(s), j\in \mathsf{T}(i(x))]\\
        v\in  \mathsf{T}(s) \rightarrow \mathsf{U}_0 \qquad\ \
      a\in \mathsf{T}(s)\qquad j\in \mathsf{T}(i(a))\\
       q\in a\ltimes_{s,i,c}v \end{array} }
  {\displaystyle \mathsf{ax_2}(a,i,q)\in (\Sigma y\in \mathsf{T}(s))\ (\ y\,\epsilon\, c(a,j)\ \times\ y\ltimes_{s,i,c}v )  \  }}
\\[20pt]
\mbox{cind-Pos} 
\displaystyle{ \frac{ 
    \begin{array}{l}     s\in \mathsf{U}_0 \ \ \ \  i(x)\in \mathsf{U}_0\  [x\in \mathsf{T}(s)] \ \ \ c(x,j)\in \mathsf{T}(s)\rightarrow  \mathsf{U}_0\ [x\in \mathsf{T}(s), j\in \mathsf{T}(i(x))]\\[3pt]
      P(x)\ type\ [x\in \mathsf{T}(s)]\qquad
      v\in \mathsf{T}(s)\ \rightarrow\  \mathsf{U}_0\qquad \\[3pt]
      a\in \mathsf{T}(s)\qquad m\in P(a)\\[3pt]
      q_1(x,z)\in \mathsf{T}(v(x))\ [ x\in \mathsf{T}(s), z\in P(x)]\qquad \qquad \\[3pt]
      q_2(x,j, z) \in(\Sigma y\in \mathsf{T}(s))  ( \ y\,\epsilon\, c(x,j)\times P(y) )\ 
      [ x\in \mathsf{T}(s), j\in \mathsf{T}(i(x)),  z\in P(x)]    \end{array} }{\mathsf{ax_3}(m, q_1,q_2) \in  a\ltimes_{s,i,c}v}}
  \end{array}\]
where $a\ltimes_{s,i,c} v$ is a shorthand for $\mathsf{T}(a\widehat{\ltimes}_{s,i,c} v)$  and we used the abbreviation $t\,\epsilon\, s$ for $\mathsf{T}(\mathsf{Ap}(s,t))$.

We can interpret the theory \mttcind\ (which has a universe of small propositions \`a la Russell) in \mlucf\ (which has a universe \`a la Tarski) as follows:
\begin{enumerate}
\item We first consider a pre-syntax of \mttcind\ consisting of pre-types and of pre-terms with the upshot that  pre-terms include also all the pre-types  contained in the universe of small propositions $\mathsf{prop}_s$; 
\item We define by mutual recursion on the complexity of the pre-syntax two interpretation functions; the first $\left\|-\right\|_t$ maps each pre-term of \mttcind\ to a pre-term of \mlucf, the second $\left\|-\right\|_T$ maps each pre-type of \mttcind\ to a pre-type of \mlucf; these interpretations are defined in the obvious way by simply following the proposition-as-types interpretation and having care of interpreting pre-types seen as pre-terms in code-terms corresponding to the relative pre-type interpretations. 
\end{enumerate}
The pre-type $\mathsf{prop}_s$ will be interpreted as the universe $\mathsf{U}_0$, and types and terms introduced in the positivity coinduction rules of \mttcind\ are interpreted in the obvious way using the corresponding types and terms in \mlucf.
Judgements of \mttcind\ are interpreted in judgements of $\mlucf$ in the obvious way, by using $\left\|-\right\|_t$ or $\left\|-\right\|_T$ when the role of a small proposition in a judgement is that of a term or that of a type, respectively. For example if $\varphi$ is a small proposition in \mttcind\, the judgement $a\in\varphi$ is translated into $\left\|a\right\|_t\in \left\|\varphi\right\|_T$, while the judgement $\varphi\in \mathsf{prop}_s$ is translated into $\left\|\varphi\right\|_t\in \mathsf{U}_0$.

\begin{prop}\label{inmtt}  We can interpret   \mttcind$ + \ac +\ct $ into \mlucf$+\ct$ according to the above interpretation so that  the consistency of \mttcind\ $+\ac +\ct$ is reduced to the consistency of $\mlucf+\ct$. \end{prop}

 \subsection{The theory \mlsi\ }
  
  We here briefly describe the theory \mlsi\  obtained by extending the first-order fragment of intensional Martin-L{\"o}f's type theory in~\cite{PMTT} with
 a superuniverse $\mathcal{S}$ \`a la Tarski closed under \emph{inductive covers} besides the usual first-order
 type constructors and universe constructors  as in~\cite{suppal}.
 
  In accordance with Curry-Howard's proposition-as-sets interpretation, which is a peculiarity of   Martin-L{\"o}f's type theory,
as done for  \mluif\  in~\cite{mmr21}  and contrary to \mttind\ as well as to \mttcind,  we strengthen the elimination
 rule of inductive basic covers to act towards sets depending on their proof-terms according to inductive generation of types in Martin-L{\"o}f's type theory.

 To this purpose we add to \mlsi\  the code
 \[ a\, \widehat{\triangleleft}_{s,i,c}\, v\in \mathcal{S}\qquad \mbox{ for } a\in \mathsf{T}(s) \mbox{ and }
 v\in \mathsf{T}(s)\ \rightarrow \ \mathcal{S}\]
 meaning that \emph{the element $a$ of a small set  $\mathsf{T}(s)$  represented by the code $s\in \mathcal{S}$ is covered by the subset
 $v$} represented by a small propositional function from $\mathsf{T}(s)$
 to the (large) set of small propositions identified with $\mathcal{S}$ by the propositions-as-sets correspondence.



 Moreover, we use $\mathsf{axcov}(s,i,c)$ to denote collectively the following judgements
  \[s\in \mathcal{S} \ \ \ \  i(x)\in \mathcal{S} \  [x\in \mathsf{T}(s)] \ \ \
  c(x,j)\in \mathsf{T}(s)\rightarrow \mathcal{S} \ [x\in \mathsf{T}(s),j\in \mathsf{T}(i(x))]\]
 
  Then, the precise rules of inductive basic covers in  \mlsi\ are obtained by those of \mluif\  in~\cite{mmr21} by replacing $\mathsf{U}_0$ with $\mathcal{S}$
  as follows:\\

 {\bf Rules of inductively generated basic covers in \mlsi}
  \nopagebreak\bigskip\nopagebreak
  \[\begin{array}{l}
\mbox{\rm F-$\triangleleft$} \
\displaystyle{\frac{\begin{array}{l}
    \mathsf{axcov}(s,i,c)\qquad 
      a\in \mathsf{T}(s)\qquad  v\in \mathsf{T}(s)\rightarrow \mathcal{S}\ \end{array} }
  {\displaystyle a\,\widehat{\triangleleft}_{s,i,c}\, v\in \mathcal{S} }}\\
  \\[20pt]
  \end{array}\]
  \[\begin{array}{l}
\mbox{\rm rf-$\triangleleft$} \
\displaystyle{\frac{\begin{array}{l}
     \mathsf{axcov}(s,i,c)\qquad 
   a\in \mathsf{T}(s)\qquad  v\in \mathsf{T}(s)\rightarrow \mathcal{S}\qquad\  r\in a\,\epsilon\, v\ \end{array} }
  {\displaystyle\mathsf{rf}(a,r)\in a \triangleleft_{s,i,c} v}}\\
\\[20pt]
  \end{array}\]
  \[\begin{array}{l}
\mbox{\rm tr-$\triangleleft$} \
\displaystyle{
  \frac{\begin{array}{l} \mathsf{axcov}(s,i,c)\qquad
      a\in \mathsf{T}(s)\qquad j\in \mathsf{T}(i(a))\qquad   v\in \mathsf{T}(s)\rightarrow \mathcal{S} \\
      r\in (\Pi z\in \mathsf{T}(s))(z\,\epsilon\,c(a,j)\rightarrow z \triangleleft_{s,i,c}v) \end{array} }
  {\displaystyle \mathsf{tr}(a,j,r) \in a\triangleleft_{s,i,c}v }}\\
\\[10pt]
  \end{array}\]
      \[\begin{array}{l}
        \mbox{ind-$\triangleleft$} 
\displaystyle{\frac{ 
    \begin{array}{l}    \mathsf{axcov}(s,i,c)\qquad      v\in \mathsf{T}(s)\rightarrow \mathcal{S} \qquad P(x,u)\  type\ [x\in \mathsf{T}(s), u\in x\triangleleft_{s,i,c}v]\qquad \\
     a \in \mathsf{T}(s) \qquad m\in a\triangleleft_{s,i,c}v\\
          q_1 (x,w)\in P(x,\mathsf{rf}(x,w))\ [x\in \mathsf{T}(s), w\in  x\,\epsilon\, v]\\[3pt]
          q_2(x,h ,k,f)\in P(x,\mathsf{tr}(x,h,k))\ \ \\
          \quad \qquad  [x\in \mathsf{T}(s), h\in \mathsf{T}(i(x)),\\
            \qquad \qquad  k\in (\Pi z\in \mathsf{T}(s))(z\,\epsilon\,c(x,h)\rightarrow z \triangleleft_{s,i,c}v),\\
           \qquad \qquad\qquad f\in
            (\Pi z\in \mathsf{T}(s))(\Pi u\in z\,\epsilon\,c(x,h))\, P(z, \mathsf{Ap}(\mathsf{Ap}(k,z),u))]       
  \end{array} }{\mathsf{ind}(m, q_1,q_2)\in  P(a,m)     }}\\[20pt]
      \end{array}\]
\[  \begin{array}{l}
  \mbox{C$_1$-ind-$\triangleleft$} 
\displaystyle{\frac{ 
    \begin{array}{l}   \mathsf{axcov}(s,i,c)\qquad 
        v\in \mathsf{T}(s)\rightarrow \mathcal{S} \qquad P(x,u)\  type\ [x\in \mathsf{T}(s),u\in x\triangleleft_{s,i,c}v ]\\
     a \in \mathsf{T}(s) \qquad r\in a\,\epsilon\, v\\
          q_1 (x,w)\in P(x,\mathsf{rf}(x,w))\ [x\in \mathsf{T}(s), w\in  x\,\epsilon\, v]\\[5pt]
          q_2(x,h,k,f  )\in P(x,\mathsf{tr}(x,h,k))\ \ \\
        \quad \qquad  [x\in \mathsf{T}(s), h\in \mathsf{T}(i(x)),\\
            \qquad \qquad  k\in (\Pi z\in \mathsf{T}(s))(z\,\epsilon\,c(x,h)\rightarrow x \triangleleft_{s,i,c}v), \\
            \qquad \qquad\qquad f\in
            (\Pi z\in \mathsf{T}(s))(\Pi u\in z\,\epsilon\,c(x,h))\, P(z, \mathsf{Ap}(\mathsf{Ap}(k,z),u))]       
  \end{array} }{\mathsf{ind}(\mathsf{rf}(a,r), q_1,q_2)= q_1(a,r )\in  P(a,\mathsf{rf}(a,r))   }}\\[20pt]
\end{array}\]
\[  \begin{array}{l}
\mbox{C$_2$-ind-$\triangleleft$} 
\displaystyle{\frac{ 
    \begin{array}{l}  \mathsf{axcov}(s,i,c)\qquad 
        v\in \mathsf{T}(s)\rightarrow \mathcal{S} \qquad P(x,u)\  type\ [x\in \mathsf{T}(s),u\in x\triangleleft_{s,i,c}v ]\\
      a \in \mathsf{T}(s) \qquad j\in \mathsf{T}(i(a))  \qquad
       r\in (\Pi z\in \mathsf{T}(s))(z\,\epsilon\,c(a,j)\rightarrow z \triangleleft_{s,i,c}v)\\
          q_1 (x,w)\in P(x,\mathsf{rf}(x,w))\ [x\in \mathsf{T}(s), w\in  x\,\epsilon\, v]\\[5pt]
          q_2(x,h ,k,f  )\in P(x,\mathsf{tr}(x,h,k))\ \ \\
        \quad \qquad  [x\in \mathsf{T}(s), h\in \mathsf{T}(i(x)),\\
            \qquad \qquad  k\in (\Pi z\in \mathsf{T}(s))(z\,\epsilon\,c(x,h)\rightarrow z \triangleleft_{s,i,c}v),\\
            \qquad \qquad\qquad f\in
            (\Pi z\in \mathsf{T}(s))(\Pi u\in z\,\epsilon\,c(x,h))P(z, \mathsf{Ap}(\mathsf{Ap}(k,z),u))]         \end{array} }{\mathsf{ind}(\mathsf{tr}(a,j,r), q_1,q_2)= q_2(a,j,r,\lambda z.\lambda u.\mathsf{ind}(\mathsf{Ap}(\mathsf{Ap}(r,z),u),q_1,q_2))\in  P(a,\mathsf{tr}(a,j,r))     }}\\[20pt]
\end{array}\]

 It is worth stressing that universes within the superuniverse are \emph{not necessarily
 closed under inductive covers}.

\begin{defi}\label{xixi}
A  crucial deviation from the ordinary versions of Martin-L\"of's type theory
is that for $\mlsi$, and also for $\mlucf$,  we \emph{postulate just the
  replacement rule repl}
  \[
\begin{array}{l}
      \mbox{repl)} \ \
\displaystyle{ \frac
         { \displaystyle 
\begin{array}{l}
 c(x_1,\dots, x_n)\in C(x_1,\dots,x_n)\ \
 [\, x_1\in A_1,\,  \dots,\,  x_n\in A_n(x_1,\dots,x_{n-1})\, ]   \\[2pt]
a_1=b_1\in A_1\ \dots \ a_n=b_n\in A_n(a_1,\dots,a_{n-1})
\end{array}}
         {\displaystyle c(a_1,\dots,a_n)=c(b_1,\dots, b_n)\in
 C(a_1,\dots,a_{n})  }}
\end{array}     
\]
  \emph{in place of the usual congruence rules which would include the $\xi$-rule} in accordance
with the rules of \mtt\ in~\cite{m09} (see~\cite{IMMS} for further details on this point).
\end{defi}

We employ this restriction in \mttcind, $\mlucf$ and  in \mlsi\  because 
the realizability semantics we present in the next sections, based on the original Kleene realizability in~\cite{DT88}, does not validate 
the $\xi$-rule\footnote{Notice that a trivial instance of the $\xi$-rule is derivable from repl) when $c$ and $c'$ do not depend on $x^B$.} of lambda-terms
\[
\mbox{  $\xi$} \
\displaystyle{\frac{ \displaystyle c=c'\in  C\ [x\in B]  }
{ \displaystyle \lambda x^{B}.c=\lambda x^{B}.c' \in (\Pi x\in B) C}}
\]
which is instead valid in~\cite{PMTT}.


Moreover, observe that the lack of the $\xi$-rule  does not affect the possibility of adopting \mtt\ as
the intensional
level of a  two-level constructive foundation
as intended in~\cite{mtt}, considered that
  the term equality rules  of \mttcind\ suffice to interpret
the extensional level \emttcind\ including extensionality of functions, 
 by means of a quotient model as that introduced in~\cite{m09} and studied abstractly
in~\cite{elqu,qu12,uxc}.

Now note that we can interpret \mttcind\    without Positivity relations, namely  \mttind\ in~\cite{mmr21}, within \mlsi\   by following  the same strategy adopted for interpreting \mttcind\ in \mlucf\ 
namely by defining  a pair of functions $\left\|-\right\|_t$ and $\left\|-\right\|_T$ as for Proposition~\ref{inmtt}.

Now we show that  in \mlsi\  we can define  positivity relations of \mlucf\   where the universe $\mathsf{U}_0$ is substituted
  with the superuniverse $\mathcal{S}$. In this way we can interpret in \mlsi\ not only \mttind\ but also \mttcind.

To increase readability and to avoid an heavy use of codes, we assume the following conventions: we call any type $A$ a \emph{small set}  if $A\equiv T(c)$ for some $c \in \mathcal{S}$.
Similarly,  we  say that  $V$ is a subset of $A$  and we write $V \subseteq A$ for any propositional function $V\in A\ \rightarrow \mathcal{S}$
derivable in \mlsi.  In addition, given two subsets $V,W$ of $A$, we  write $V\subseteq W$ if we have a proof-term $p$ for which we can derive
$p\in \forall_{x\in A}\ (\ V(x)\ \rightarrow \ W(x)\ )$. We also  adopt the abbreviation $\{\ x\in A \  \mid\ \phi(x) \ \}$ 
to indicate a propositional function $\phi\in A\ \rightarrow \mathcal{S}$.
 
 Moreover, given a family of small sets $B(x) \ type\ [x\in A]$ on a small set $A$
we denote by  $ U(A,B(x))$ the universe in $\mathcal{S}$ containing them.
 
Finally, we also  call  \emph{axiom-set }  the judgements
  \[A \ type \ \ \ \  I(x)\  type \  [x\in A] \ \ \ C(x,j) \ type\ [x\in A, j\in I(x)]\]
 if and only if they are derived from an axiom-set  $\mathsf{axcov}(s,i,c)$
 in the sense that $A=T(s)$, $I(x) =T(i(x))$  and $C(x,j)= T(c(x,j))$ for $x\in A$ and $j\in I(x)$.

\begin{thm}\label{coipos}
In \mlsi\ we can define
  Coinductive Positivity relations satisfying the instances of the  rules in \mlucf\  at the beginning of Subsection~\ref{mlcpos} where the universe $\mathsf{U}_0$ is substituted
  with the superuniverse $\mathcal{S}$.
     \end{thm}
        {\bf Proof.} We adapt here  an  argument in the appendix of~\cite{mv04} originally due to Thierry Coquand. For sake of readability and thanks to the notational conventions just established, we will proceed  as if we were working  with a superuniverse \`a la Russell. 
        
%

For any  axiom-set   $A \ type$ with $ I(x)\  type \  [x\in A] $ and $C(x,j) \ type\ [x\in A, j\in I(x)]$,  made all of small sets and small families,
 and for any fixed subset  $V \subseteq A$ and $a\in A$ the positivity relation amounts to be defined as 
 \[a\ \ltimes V \equiv\, a\,\epsilon\, W_{\max}(V)\in  \mathcal{S}\]
where $W_{\max}(V) \in  \mathcal{S}$ is 
 the maximal fixpoint of an operator $\tau\in \mathcal{P}_{int}(V) \to \mathcal{P}_{int}(V)$ on the \emph{intensional} representation  of the powerset   of the subset $V$
 \[\mathcal{P}_{int}(V) \equiv \Sigma_{x\in A} V(x)\ \rightarrow \mathcal{S}\]
defined by setting
\[
\tau(X)
\ \equiv
\ \{\  x \in A \ \mid\ x\,\epsilon\, X \ \& \ 
\forall_{i \in I(x)}\ \exists_{y\in A}\ (\ y \,\epsilon\, C(x,i)\ \&\ y\,\epsilon\, X \ )\ \}
\]
for any subset $X$ of $A$ within $V$ which preserves  extensional equality of subsets as  $=_{ext}$ of Lemma~\ref{prepos}~\footnote{The operator $\tau$
is clearly monotone and hence in an impredicative classical foundation,  by Tarski fixpoint theorem,
it admits a maximal fixed point  $W_{\max}(V)$  defined as
\emph{the union of all the subsets $Y$ of $A$ within $V$ such that
$Y \subseteq \tau(Y)$} after noting that such a 
family of subsets is not empty,  since the empty subset $\emptyset$ satisfies the condition
trivially.}.

Now, observe that there exists a universe $U_{I,C}$ in the superuniverse $\mathcal{S}$ containing  all the components of the axiom-set.
%
%
%
Then  the biggest fixed point can be defined  in \mlsi\  as an element of $A\ \rightarrow\  \mathcal{S}$ as follows
\[ W_{\max}(V) \ \equiv\ \{\ x\in A\ \mid\ 
\exists\ _{Z\in A\ \rightarrow \ U_{I,C}}\ (\ x \,\epsilon\, Z\ \& \ ( \ Z \subseteq \tau(Z)\ \ \&\ Z \subseteq V\ )\ )\  \}\]
since $\mathcal{S}$  contains $U_{I,C}$
as well as $Z$ and $V$.


This is really the biggest fixed point since for any fixed subset
 $Y$ of $A$ within $V$  such that $Y \subseteq \tau(Y)$
and any
fixed $\overline{a}\in A$ such that 
$\overline{a} \,\epsilon\, Y$, we have that  $\overline{a}\,\epsilon\, W_{\max}(V)$ holds, namely
that
$Y\subseteq W_{\max} (V)$.

To this purpose by an application of the so called \emph{axiom of choice} of Martin-L{\"o}f's type theory
to the formalization of $Y \subseteq \tau(Y)$
\[
\forall_{x \in A}\ (\ x \,\epsilon\, Y\  \to 
\  \forall_{j\in I(x)}\ \exists_{w \in A}\ (\ w \,\epsilon\, C(x,j)\ \&\ w \,\epsilon\, Y\ ) \ )
\]
we derive  the existence of a choice function $f_x$ for any $x\in A$ as follows
\[
\forall_{x \in A}\ (\ x \,\epsilon\, Y\  \to \
\ \exists_{{f_x} \in I(x) \to A }\ \ \forall_{j \in I(x)}\ (\ f_x(j) \,\epsilon\, C(x,j)\ \& \ f_x(j) \,\epsilon\, Y\ )\  )\]

Then, 
we define  by induction on natural numbers the following
sequence of subsets of $A$:
\[
\begin{array}{rcl}
X_0 &\equiv &\{\ w\in A\ \mid\ w=_A\overline{a}\   \}
\\
X_{n+1} &\equiv 
&X_n \cup \{\  w \in\  A \mid \exists_{x\in A}\ \exists_{j\in I(x)} (\ (\ x\,\epsilon\, X_n \ \&\  w=_{A} f_x(j)\ \}
\end{array}
\]

Observe that the subset of $A$ defined as
\[
X^{\ast} \equiv \bigcup_{n \in \mathsf{Nat}} X_n
\]
is in $U_{I,C}$, since  $U_{I,C}$ is closed under all first-order type constructors, and $X^{\ast}\ \subseteq \ Y$.


Finally observe that $X^{\ast} \subseteq \tau(X^{\ast})$ is true in \mlsi\ and hence $X^\ast \subseteq W_{\max}(V)$ so that we can conclude $\overline{a}\ \ltimes V \equiv\, \overline{a}\,\epsilon\, W_{\max}(V)\in  \mathcal{S}$ as claimed.

\medskip

 
By following  the same strategy adopted for Proposition~\ref{inmtt} we can conclude:
\begin{cor}\label{cindinsup}
   \mttcind\  as well as \mlucf\ can be interpreted in \mlsi\ by interpreting  coinductive positivity relations as in Theorem~\ref{coipos}.
   In particular,    the consistency of  \mttcind\ $+\ac +\ct$ is reduced to the consistency of $\mlsi+\ct$.   \end{cor}

\section{Inductive and coinductive definitions in Aczel's CZF}
In the following we shall introduce several inductively defined classes, and, moreover,
we have to ensure that such classes
can be formalized in $\CZF$.

\begin{defi}\label{MR1}
We define an \emph{inductive definition} to be a class of ordered pairs.
If $\Phi$ is an inductive definition and $\langle x,a\rangle\in\Phi$ then
we write
    \begin{eqnarray*} &&\step{x}{a}{_\Phi}\end{eqnarray*}
and call $\step{x}{a}{}$ an \emph{(inference) step} of $\Phi$, with set
$x$ of \emph{premisses} and \emph{conclusion} $a$.  For any class $Y$, let
\begin{eqnarray*} \Gamma_{\Phi}(Y)&=& \bigl\{a\,\mid\;\exists x\,\bigl(x\subseteq Y\;\;\wedge\;\;
\step{x}{a}{_\Phi}\,\bigr)\bigr\}.\end{eqnarray*}
 The class $Y$ is \emph{$\Phi$-closed} if $\Gamma_{\Phi}(Y)\subseteq Y$.
Note that $\Gamma$ is monotone; i.e.\ for classes $Y_1,Y_2$, whenever $Y_1\subseteq Y_2$, then
$\Gamma(Y_1)\subseteq \Gamma(Y_2)$.

We define the class \emph{inductively defined by $\Phi$} to be the
smallest $\Phi$-closed class.  
\end{defi}

The main result about inductively defined classes states
that this class, denoted $\I(\Phi)$, always exists.

\begin{thm}\label{ind} $(\CZF)$ \emph{(Class Inductive Definition Theorem)}
For any inductive definition $\Phi$ there is a smallest $\Phi$-closed
class $\I(\Phi)$.
\end{thm}
\begin{proof}
~\cite{aczel86} section 4.2 or~\cite{czf} Theorem 5.1 or~\cite{czf2} Theorem 12.1.1. \end{proof} 

A similar result can be obtained for the dual notion of largest $\Phi$-closed class. However, we need to enlist a choice principle.

\begin{defi}\label{MR2}  The
 \emph{Relativized Dependent Choices Axiom}, $\mathbf{RDC}$, is the following scheme.
 For arbitrary formulae $\phi$ and
$\psi$, whenever
  \[\forall x\bigl[\phi(x)\,\rightarrow\,\exists y
  \bigl(\phi(y)\,\wedge\,\psi(x,y)\bigr)\bigr]\] and $\phi(b_0)$, then
there exists a function $f$ with domain $\omega$ such that $f(0)=b_0$ and
  \[(\forall n\in\omega)\bigl[\phi(f(n))\,\wedge\,\psi(f(n),f(n+1))\bigr].\]
\end{defi}

\begin{thm}\label{coind}  $(\CZF+\mathbf{RDC})$ \emph{(Class Co-Inductive Definition Theorem)}
For any inductive definition $\Phi$ there is a largest $\Phi$-closed
class $\CoI(\Phi)$.
\end{thm}
\begin{proof}
~\cite{aczel88}, Theorem 6.5 (classically)  or~\cite{circrat} Theorem 5.17 or~\cite{czf2}  Theorem 13.1.3 in combination with Proposition 13.1.2. \end{proof} 

We are mostly interested in the case when $\I(\Phi)$ and $\CoI(\Phi)$ are sets. This requires some stronger axioms, though.

\begin{defi} A  set $A$ is a \emph{regular set} if it is transitive and inhabited such that whenever $a\in A$ and $R\subseteq A\times A$ is a relation satisfying
$\forall x\in A\, \exists y\in A\;xRy$ then there exists a set $c\in A$ such that
  \[\forall x\in a\,\exists y\in c\,xRy\;\;\wedge\;\;\forall y\in c\,\exists x\in a \,xRy.\]
The set $A$ is said to be strongly regular or $\bigcup$-regular if it is regular and $\forall a\in A\;\bigcup a\in A$. 

$\mathbf{REA}$ is the assertion that every set is contained in a regular set. $\mathbf{REA}_{\bigcup}$ is the assertion that every set is contained in a $\bigcup$-regular set. 
\end{defi}

\begin{thm}\label{indset}  $(\CZF+\mathbf{REA})$  Let $\Phi$ be an  inductive definition such that $\Phi$ is a set,
then $\I(\Phi)$ is a set.
\end{thm}
\begin{proof}~\cite{aczel86} or~\cite{czf} Theorem 5.7 or~\cite{czf2} Theorem 12.2.4.\end{proof}

\begin{thm} \label{coindczf} $(\CZF+\mathbf{REA}_{\bigcup}+\mathbf{RDC})$  Let $\Phi$ be an  inductive definition such that $\Phi$ is a set,
then $\CoI(\Phi)$ is a set.
\end{thm}
\begin{proof} This follows from~\cite{czf2} Theorem 13.2.3 in combination with~\cite{czf2}  Proposition 13.1.2.
More precisely,~\cite{czf2} Theorem 13.2.3 uses the notion of an $\mathbf{RRS}$ strongly regular set and assumes the axiom that every set is contained in an
$\mathbf{RRS}$ strongly regular set. In the presence of $\mathbf{RDC}$, however, one can show that every $\bigcup$-regular set which contains $\omega$ as an element
is already  an $\mathbf{RRS}$ strongly regular set. The latter follows basically by  the same argument as in the proof of~\cite{czf2}  Proposition 13.1.2. \end{proof}

\section{Realizability interpretations for \mlucf\ and \mlsi\ }           
\subsection{A realizability interpretation of $\mlucf$ with $\mathbf{CT}$ in $\mathbf{CZF}+\mathbf{REA}_{\cup}+\mathbf{RDC}$}   \label{sec-real}      

Here we are going to describe 
a realizability model of $\mlucf$ with \ct\
extending that in~\cite{mmr21} in the constructive theory $\mathbf{CZF}+\mathbf{REA}_{\cup}+\mathbf{RDC}$.

 As per usual in set theory,  we identify the natural numbers with the finite ordinals, i.e.\ $\mathbb{N}:=\omega$. To simplify the treatment we will assume that $\mathbf{CZF}$ has names for all (meta) natural numbers. Let $\overline{n}$ be the constant designating the $n^{th}$ natural number. We also assume that $\mathbf{CZF}$ has function symbols for addition and multiplication on $\mathbb{N}$ as well as for a primitive recursive bijective pairing function $\pp:\mathbb{N}\times \mathbb{N}\rightarrow \mathbb{N}$ and its primitive recursive inverses $\pp_0$ and $\pp_1$, that satisfy $\pp_0(\pp(n,m))=n$ and $\pp_1(\pp(n,m))=m$. We also assume that $\mathbf{CZF}$ is endowed with symbols for a primitive recursive length function $\ell:\mathbb{N}\rightarrow \mathbb{N}$ and a primitive recursive component function $(-)_{-}:\mathbb{N}\times \mathbb{N}\rightarrow \mathbb{N}$ determining a bijective encoding of finite lists of natural numbers by means of natural numbers. $\mathbf{CZF}$ should also have a symbol $T$ for Kleene's $T$-predicate and the result extracting function $U$. Let $P(\{e\}(n))$ be a shorthand for $\exists m(T(e,n,m)\wedge P(U(m)))$. Further, let $\pp(n,m,k):=\pp(\pp(n,m),k)$, $\pp(n,m,k,h):=\pp(\pp(n,m,k),h)$, etc.  We will use $\pp^n_i$ (for $0<i<n$) to denote the $i$th component function for which $\pp^n_i(\pp(m_0,\ldots,m_{n-1}))=m_i$ for every $m_0,\ldots,m_{n-1}\in \mathbb{N}$.
 A similar convention will be adopted for application of partial recursive functions: Let $\{e\}(a,b):=\{\{e\}(a)\}(b)$,  $\{e\}(a,b,c):=\{\{e\}(a,b)\}(c)$ etc.
 We use $a,b,c,d,e,f,n,m,l,k,q,r,s,v,j,i$ as metavariables for natural numbers.  

We first need to introduce some abbreviations: 
\begin{enumerate}
\item $\mathsf{n}_0$ is $\pp(0,0)$, $\mathsf{n}_1$ is $\pp(0,1)$ and $\mathsf{n}$ is $\pp(0,2)$.
\item $\sigmat(a,b)$ is $\pp(1,\pp(a,b))$, $\pit(a,b)$ is $\pp(2,\pp(a,b))$ and $+(a,b)$ is $\pp(3,\pp(a,b))$.
\item $\mathsf{list}(a)$ is $\pp(4,a)$ and $\mathsf{id}(a,b,c)$ is $\pp(5,\pp(a,b,c))$.
\item $a\widetilde{\triangleleft}_{c,d,e}b$ is $\pp(6,\pp(a,b,c,d,e))$.
\item $\rft(a,r)$ is  $\pp(7,\pp(a,r))$.
\item $\trt(a,j,r)$ is $\pp(8,\pp(a,j,r))$. 
\item $a\widetilde{\ltimes}_{c,d,e} b$ is  $\pp(9,\pp(a,b,c,d,e))$.
\end{enumerate}

Recall that in intuitionistic set theories ordinals are defined as transitive sets all of whose members are transitive sets, too. Unlike in the classical case, one cannot prove that they are linearly ordered but they are perfectly good as a scale along which one can iterate various processes. The trichotomy of $0$, successor, and limit ordinal, of course, has to be jettisoned. As usual, we use lowercase Greek letters as metavariables for ordinals.

\begin{defi}\label{fixdef} By transfinite recursion on ordinals (cf.~\cite{czf2}, Proposition 9.4.4) we define simultaneously two relations $\mathsf{Set}_\alpha(n)$ and $n\,\varepsilon _\alpha\,m$ on $\mathbb{N}$ in $\mathbf{CZF}+\mathbf{REA}_{\cup}+\mathbf{RDC}$.

In the following definition we use the shorthand $\mathsf{Fam}_\alpha(e,k)$ to convey that $\mathsf{Set}_\alpha(k)$ and $\forall j(j\,\varepsilon _\alpha\,k\rightarrow \mathsf{Set}_\alpha(\{e\}(j)))$ and we shall write $\mathsf{Set}_{\in \alpha}(n)$ for $\exists \beta\in \alpha(\mathsf{Set}_\beta (n))$, $n\,\varepsilon _{\in\alpha}\,m$ for $\exists \beta\in \alpha(n\,\varepsilon _\beta\,m)$ and $\mathsf{Fam}_{\in\alpha}(e,k)$ for $\exists \beta\in \alpha(\mathsf{Fam}_\beta(e,k))$.
\begin{enumerate}
\item[]
\item $\mathsf{Set}_\alpha(\mathsf{n}_j)$ if $j=0$ or $j=1$, and $m\,\varepsilon _\alpha\,\mathsf{n}_j$ if $m<j$;
\item $\mathsf{Set}_\alpha(\mathsf{n})$ holds, and $m\,\varepsilon _\alpha\,\mathsf{n}$ if $m\in \mathbb{N}$.
\item[]
\item If $\mathsf{Fam}_{\in \alpha}(e,k)$, then $\mathsf{Set}_\alpha(\pit(k,e))$ and $\mathsf{Set}_\alpha(\sigmat(k,e))$;
\item[] if $\mathsf{Fam}_{\in \alpha}(e,k)$, then
\begin{enumerate}
\item $n\,\varepsilon _\alpha\,\pit(k,e)$ if there exists $\beta\in \alpha$ such that $\mathsf{Fam}_{\beta}(e,k)$ and 
$(\forall i\,\varepsilon _\beta\,k)\, \{n\}(i)\,\varepsilon_\beta\, \{e\}(i)$.\footnote{We use the obvious shorthand  $(\forall i\,\varepsilon _\beta\,k) \ldots$ for 
$\forall i[i\,\varepsilon_\beta\,k\rightarrow\ldots]$; also employed henceforth.} 
\item $n\,\varepsilon_\alpha\,\sigmat(k,e)$ if there exists $\beta\in \alpha$ such that $\mathsf{Fam}_{\beta}(e,k)$ as well as $\pp_0(n)\,\varepsilon_\beta\,k$ and $\pp_1(n)\,\varepsilon_\beta\, \{e\}(\pp_0(n))$.
\end{enumerate}
\item[]
\item If there exists $\beta\in\alpha$ such that $\mathsf{Set}_\beta(n)$ and $\mathsf{Set}_\beta(m)$, then $\mathsf{Set}_\alpha(+(n,m))$, and 
\item[] $i\,\varepsilon _\alpha {+}(n,m)$ if there exists $\beta\in\alpha$ such that $\mathsf{Set}_\beta(n)$, $\mathsf{Set}_\beta(m)$ and 
  \[[\pp_0(i)=0\wedge \pp_1(i)\,\varepsilon_\beta\,n]\,\vee\,[\pp_0(i)=1\wedge \pp_1(i)\,\varepsilon_\beta\,m].\]
\item[]
\item  If there exists $\beta\in\alpha$ such that $\mathsf{Set}_\beta(n)$, then $\mathsf{Set}_\alpha(\mathsf{list}(n))$, and 
\item[] $i\,\varepsilon_\alpha\, \mathsf{list}(n)$ if there exists $\beta\in\alpha$ such that $\mathsf{Set}_\beta(n)$ and $\forall j\,[j<\ell(i)\rightarrow (i)_j \,\varepsilon_\beta\,n]$.
\item[]
\item If $\mathsf{Set}_{\in \alpha}(n)$, then $\mathsf{Set}_\alpha(\mathsf{id}(n,m,k))$, and 
\item[] $s\,\varepsilon_\alpha\,\mathsf{id}(n,m,k)$ if there exists $\beta\in \alpha$ such that $\mathsf{Set}_{\beta}(n)$, $m\,\varepsilon_\beta\,n$ and $s=m=k$.
\item[]
\item Let $\beta\in \alpha$. Suppose that the following conditions (collectively called $*_\beta$) are satisfied: 
\begin{enumerate}
\item $\mathsf{Set}_\beta(s)$, 
\item $a\,\varepsilon _\beta\,s$, 
\item $\mathsf{Fam}_\beta(v,s)$,  
\item $\mathsf{Fam}_\beta(i,s)$ and
\item $\forall x\forall y[x\,\varepsilon _\beta s\,\wedge\, y\,\varepsilon _\beta \{i\}(x)\rightarrow \mathsf{Fam}_\beta( \{c\}(x,y),s)]$. 
\end{enumerate}
 Then $\mathsf{Set}_\alpha(a\widetilde{\triangleleft}_{s,i,c} v)$
\item[] For $\beta\in\alpha$ satisfying $*_\beta$, let $V_{\beta}$ be the smallest subset of $\mathbb{N}$ satisfying the following conditions:
\begin{enumerate}
\item if $z\, \varepsilon _\beta s$ and $r \,\varepsilon _\beta  \{v\}(z)$ then $\pp(z,\rft(z,r))\in V_{\beta}$;
\item if $r\in\mathbb{N}$, $z\, \varepsilon_\beta\, s $,  $j \,\varepsilon_\beta  \{i\}(z)$ 
and $(\forall u\, \varepsilon_\beta s)\,(\forall t\,\varepsilon_\beta \{c\}(z,j,u))\ \pp(u,\{r\}(u,t))\in V_\beta$ 
then $\pp(z,\trt(z,j,r))\in V_\beta$. 
\end{enumerate}
 The existence of the set $V_\beta$ is guaranteed by the fact that $\mathbf{REA}$ holds, i.e., Theorem~\ref{indset}. 
 \item[] We define $q\,\varepsilon _\alpha\,a\widetilde{\triangleleft}_{s,i,c} v$ by $\exists \beta\in \alpha\,[\,*_\beta \,\wedge\, \pp(a,q)\in V_{\beta}]$.
 \item[]
 
\item Let $\beta\in \alpha$. Suppose that the conditions $*_\beta$ as in $(7)$ are satisfied, then $\mathsf{Set}_\alpha(a\widetilde{\ltimes}_{s,i,c} v)$.
 \item[] For $\beta\in\alpha$ satisfying $*_\beta$, let $W_{\beta}$ be the largest subset of $\mathbb{N}$ satisfying the following conditions:
\begin{enumerate}
\item if $z\,\varepsilon_\beta\,s$ and $\pp(z,q)\in W_\beta$, then $\pp_0(q)\,\varepsilon_\beta\,\{v\}(z)$; 
\item if $z\,\varepsilon_\beta\,s$, $j\,\varepsilon_\beta\,\{i\}(z)$ and $\pp(z,q)\in W_\beta$, then 
  \begin{align*}
\pp_0(\{\pp_1(q)\}(j))\,\varepsilon_\beta\,s\wedge\\
 \pp_0(\pp_1(\{\pp_1(q)\}(j)))\,\varepsilon_\beta \{\{c\}(z,j)\}((\pp_0(\{\pp_1(q)\}(j)))\wedge\\
  \pp(\pp_0(\{\pp_1(q)\}(j)),\pp_1(\pp_1(\{\pp_1(q)\}(j)))) & \in W_\beta\end{align*}
\end{enumerate}
\item[]
\item[] The existence of the set $W_\beta$ is guaranteed by $\mathbf{REA}_{\cup}+\mathbf{RDC}$ (see Theorem~\ref{coindczf}). 
 \item[] We define $q\,\varepsilon _\alpha\,a\widetilde{\ltimes}_{s,i,c} v$ as $\exists \beta\in \alpha\,[\,*_\beta \,\wedge\, \pp(a,q)\in W_{\beta}]$.

\end{enumerate}
\end{defi}
%



As done in~\cite{mmr21} one can prove the following crucial lemma.

\begin{lem}\label{fix} \emph{($\mathbf{CZF}+\mathbf{REA}_{\cup}+\mathbf{RDC}$)}
\begin{itemize} 
\item For all $m\in \mathbb{N}$, if $\mathsf{Set}_\alpha(m)$ and $\alpha\subseteq \rho$, then $\mathsf{Set}_\rho(m)$.
  \item
For all $m\in \mathbb{N}$, if $\mathsf{Set}_\alpha(m)$, then for all $\rho$ such that $\mathsf{Set}_\rho(m)$, 
    \[\forall i\in \mathbb{N}(i\,\varepsilon _\alpha\,m\leftrightarrow i\,\varepsilon _\rho\,m).\]
\end{itemize}

\end{lem}

\begin{defi}We define in $\mathbf{CZF}+\mathbf{REA}_{\cup}+\mathbf{RDC}$ the formula $\mathsf{Set}(n)$ as $\exists \alpha(\mathsf{Set}_\alpha(n))$ and $x\,\varepsilon \,y$ as $\exists \alpha(x\,\varepsilon _\alpha\, y)$.

\end{defi}

\begin{thm}\label{real}
        Consistency of the theory $\mathbf{CZF}+\mathbf{REA}_{\cup}+\mathbf{RDC}$  implies the consistency of the  theory $\mlucf$ extended with the formal Church thesis \ct.
        \end{thm}
 \begin{proof} We outline a realizability semantics in $\mathbf{CZF}+\mathbf{REA}_{\cup}+\mathbf{RDC}$. In what follows $\mathbf{p}$ will be a code for the primitive recursive pairing function $\pp$ introduced just before Definition~\ref{fixdef}, that is 
$\{\mathbf{p}\}(n,m) =\pp(n,m)$. Moreover $\mathbf{p}_1$ and $\mathbf{p}_2$ will be codes for $\pp_0$ and $\pp_1$. In the same way we fix codes $\mathbf{p}^n$ and $\mathbf{p}^n_{i}$ representing the encoding of $n$-tuples and their projections, respectively. 

Every pre-term is interpreted as a $\mathcal{K}_1$-applicative term (that is, a term built with numerals, variables and Kleene application) as it is done in~\cite{mmr21}. 

We must notice that in introducing codes for sets in the universe in Definition~\ref{fixdef} we took account of dependencies by means of natural numbers representing recursive functions; however every pre-term depending  on variables will be interpreted as a $\mathcal{K}_1$-applicative term having the same free variables (we identify the variables of \mlucf\ with those in $\mathbf{CZF}+\mathbf{REA}_{\cup}+\mathbf{RDC}$). For these reasons, whenever a term $s$ in \mlucf\ depends on terms $t_1,\ldots,t_n$ in context, its interpretation will depend on the interpretations of $t_1,\ldots,t_n$ bounded with adequate $\Lambda$ operators. The variables which will be bounded by these $\Lambda$s will be the ones used in the rule where the term $s$ is introduced. This abuse of notation allows us to avoid heavy fully-annotated terms in the syntax. 

In order to complete the interpretation of pre-terms we only need to interpret the new pre-terms of $\mlucf$ as follows.
\begin{enumerate}
\item $(a\ltimes_{s,i,c}v)^I:=\{\mathbf{p}\}(9,\{\mathbf{p}^5\}(a^I,v^{I},s^I,\Lambda x.i^I,\Lambda x.\Lambda y.c^I))$
\item $\mathsf{ax}_1(a,q)^I:=\{\mathbf{p}_0\}(q^I)$
\item $\mathsf{ax}_2(a,j,q)^I:=\{\{\mathbf{p}_1\}(q^I)\}(j^I)$
\item $\mathsf{ax_3}(a,m, q_1,q_2)^I:=\{\mathbf{p}\}(q_1^I[a^I/x,m^I/z],\Lambda j.q_2^I[a^I/x,m^I/z])$ 

\end{enumerate}

\noindent 
If $\tau$ is a $\mathcal{K}_1$-applicative term and $A=\{x|\,\phi\}$ is a class, we will define $\tau\,\in\, A$ as an abbreviation for $\phi[\tau/x]$.

The interpretation of types, contexts and judgements is exactly as in~\cite{mmr21}. 

\noindent In particular, we interpret pre-types into the language of $\mathbf{CZF}+\mathbf{REA}_{\cup}+\mathbf{RDC}$ as definable subclasses of $\mathbb{N}$ 
 as follows:
\begin{enumerate}
\item $\mathsf{N}_{0}^{I}:=\{x\in \mathbb{N}|\,\bot\}$.
\item $\mathsf{N}_{1}^{I}:=\{x\in \mathbb{N}|\,x=0\}$.
\item $\mathsf{N}^{I}:=\{x\in \mathbb{N}|\,x=x\}$.
\item $((\Sigma y\in A)B)^{I}:=\{x\in \mathbb{N}|\,\pp_0(x)\in A^{I}\wedge \pp_1(x)\in B^{I}[\pp_0(x)/y] \}$.
\item $((\Pi y\in A)B)^{I}:=\{x\in \mathbb{N}|\,\forall y\in \mathbb{N}\,[y\in A^{I}\rightarrow \{x\}(y)\in B^{I}]\}$.
\item $(A+B)^{I}:=\{x\in \mathbb{N}|\,[\pp_0(x)=0 \wedge \pp_1(x)\in A^{I}]\,\vee\,[\pp_0(x)=1 \wedge \pp_1(x)\in B^{I}]\}$.
\item $(\mathsf{List}(A))^{I}:=\{x\in \mathbb{N}|\,\forall i\in \mathbb{N}\,[i<\ell(x)\rightarrow (x)_{i}\in A^{I}]\}$.
\item $(\mathsf{Id}(A,a,b))^{I}:=\{x\in \mathbb{N}|\,x=a^{I}\wedge a^{I}=b^{I}\, \wedge\,  a^{I}\in A^{I}\}$.
\item $U_0^{I}:=\{x|\,\mathsf{Set}(x)\}$.
\item $\mathsf{T}(a)^{I}:=\{x|\,x\,\varepsilon \,a^{I}\}$.
\end{enumerate}

\noindent Pre-contexts are interpreted as conjunctions of set-theoretic formulas  as follows:
\begin{enumerate}
\item $[\;]^{I}$ is the formula $\top$;
\item $[\Gamma, x\in A]^{I}$ is the formula $\Gamma^{I}\,\wedge\, x^{I}\in A^{I}$.
\end{enumerate}

\noindent Validity of judgements $J$ in $\mathbf{CZF}+\mathbf{REA}_{\cup}+\mathbf{RDC}$ under the foregoing interpretation is defined as follows:

\begin{enumerate}
\item $A\, type\,[\Gamma]$ holds if $\mathbf{CZF}+\mathbf{REA}_{\cup}+\mathbf{RDC}\vdash \Gamma^{I}\to\forall x\,(x\in A^I\rightarrow x\in \mathbb{N})$ 
\item $A=B\, type\,[\Gamma]$ holds if  $\mathbf{CZF}+\mathbf{REA}_{\cup}+\mathbf{RDC}\vdash \Gamma^{I}\to\forall x\,(x\in A^{I}\leftrightarrow x\in B^{I})$
\item $a\in A\,\,[\Gamma]$ holds if  $\mathbf{CZF}+\mathbf{REA}_{\cup}+\mathbf{RDC}\vdash \Gamma^{I}\to a^{I}\in A^{I}$
\item $a=b\in A\, \,[\Gamma]$ holds if  $\mathbf{CZF}+\mathbf{REA}_{\cup}+\mathbf{RDC}\vdash \Gamma^{I}\to a^{I}\in A^{I}\wedge a^{I}=b^{I}$,
\end{enumerate}
where $x$ is a fresh variable.

The encoding of lambda-abstraction in terms of $\mathcal{K}_1$-applicative terms can be chosen (see~\cite{IMMS}) in such a way that if $a$ and $b$ are terms and $x$ is a variable which is not bounded in $a$, then the terms $(\,a[b/x]\,)^{I}$ and $a^{I}[\,b^{I}/x^{I}]$ coincide.

The rules relative to positivity relations in \mlucf\ are satisfied by the realizability interpretation and Theorem~\ref{coindczf} plays a crucial role for the validity of the rule cind-Pos. 
\end{proof}
\subsection{A realizability interpretation of $\mlsi$ with $\mathbf{CT}$ in $\mathbf{CZF}+\mathbf{REA}$}        $\phantom{M}$

Here we are going to describe 
a realizability model of $\mlsi$ with \ct\
 in the constructive theory $\mathbf{CZF}+\mathbf{REA}$. The interpretation is analogous to that of $\mluif$ in $\mathbf{CZF}+\mathbf{REA}$ in~\cite{mmr21}, but one needs to take care of the universe constructor $\mathbf{u}(a,(x)b)$ 
 for small universes; note that the latter are  not required to be closed under inductive and co-inductive types.
 Such universes can be modelled following the construction in Definition~\ref{fixdef} but omitting the clauses (7) and (8). Whereas the transfinite recursion of Definition~\ref{fixdef} 
 has to run through all ordinals, rendering the model a proper class, one can show in $\mathbf{CZF}+\mathbf{REA}$ that there exist ordinals $\rho$ where this recursion without the clauses (7) and (8) comes to a halt, that is, no new types 
 are generated at later stages $\gamma$ when $\rho\in \gamma$. To find such $\rho$, choose a regular set $R$ that contains all the relevant types and let $\rho$ be
 $\{\xi\mid\xi\in R\}$. As a result,  small universes can be modelled via sets in one fell swoop and then their sets and the elementhood relation between them become part of the big inductive definition of the superuniverse at level $\alpha$. In other words, the small universes $U(k, e, A,(Ba) a\in A) $ are already defined in their entirety before one starts to define the large superuniverse as a class.

 To be a bit more specific, given $k\in\mathbb{N}$, a set $A\subseteq\mathbb{N}$, $e\in\mathbb{N}$ for which $\{e\}(a)$ is defined for all $a\in A$  and a family 
 $(B_a)_{a\in A}$ of subsets of $\mathbb{N}$ such that $\{e\}(a) =\{e\}(a')$ entails $B_a=B_{a'}$, let \[\mathbb{U}(k,e,A,(B_a)_{a\in A})\] be the small universe defined by the clauses (1)--(6) of Definition~\ref{fixdef} plus an initial clause to the effect that $\mathbb{U}(k,e,A,(B_a)_{a\in A})\models \mathsf{Set}(k)$,
  $\mathbb{U}(k,e,A,(B_a)_{a\in A})\models \mathsf{Set}(\{e\}(a))$ for $a\in A$, $\mathbb{U}(k,e,A,(B_a)_{a\in A})\models a\,\varepsilon\, k$ iff $a\in A$, and
   $\mathbb{U}(k,e,A,(B_a)_{a\in A})\models m\,\varepsilon\, \{e\}(a)$ iff $m\in A_a$, whenever $a\in A$; here we use $\mathbb{U}(k,e,A,(B_a)_{a\in A})\models \theta$ to express that $\theta$ holds in the sense of  $\mathbb{U}(k,e,A,(B_a)_{a\in A})$. 
  
 
\noindent To define the model for  $\mlsi$  we proceed as in Definition~\ref{fixdef}, keeping clauses (1)--(7), but replacing clause (8) as follows.
 
 \begin{enumerate}
   \item[(8')] Suppose $\mathsf{Fam}_{\in \alpha}(k,e)$. Then \[\mathsf{Set}_{\alpha}(\widetilde{\mathsf{u}}(k,e)),\] where  $\widetilde{\mathsf{u}}(k,e):= \pp(10,\pp(k,e))$.

Let $A:=\{x\mid x\,\varepsilon_{\beta}\,k\mbox{ for some }\beta\in\alpha\}$ and $B_x:=\{v\mid v\,\varepsilon_{\beta}\,\{e\}(x)\mbox{ for some }\beta\in \alpha\}$.
Since $A$ and the sets $B_a$ are determined by $k$ and $e$, we just write
$\mathbb{U}(k,e)$ for the small universe $\mathbb{U}(k,e,A,(B_x)_{x\in A})$. We also want to inject $\mathbb{U}(k,e)$ into the model for $\mlsi$. 
This is effected by the following postulations.
\begin{enumerate}
\item $d\,\varepsilon_\alpha\, \widetilde{\mathsf{u}}(k,e)$  iff $\mathbb{U}(k,e)\models \mathsf{Set}(d)$.
\item  Moreover, if  $\mathbb{U}(k,e)\models \mathsf{Set}(d)$, then $\mathsf{Set}_{\alpha}(d)$,
and $x\,\varepsilon_{\alpha}\,d$ iff  $\mathbb{U}(k,e)\models x\,\varepsilon\, d$.
\end{enumerate}

\end{enumerate}

\noindent The interpretation of type theory is then carried out in the same vein as in~\cite{mmr21}.  One only has to add the interpretation of terms and types involving the universe constructors and the superuniverse. We follow here the syntax in~\cite{suppal}, although we use $\mathcal{S}$ for the superuniverse and $\mathsf{T}$ for its decoding constructor:
\begin{enumerate}
  \item the universe constructor terms are interpreted as $(\mathbf{u}(a,(x)b))^I=^{def}\{\mathbf{p}\}(10,\{\mathbf{p}\}(a^{I},\Lambda x.b^{I}))$ and their corresponding universe types are interpreted as \[(\mathsf{U}(\mathsf{T}(a),(x)\mathsf{T}(b)))^I:=\{x|\,x\,\varepsilon\,\{\mathbf{p}\}(10,\{\mathbf{p}\}(a^{I},\Lambda x.b^{I}))\};\footnote{We do not need to interpret universe types $\mathsf{U}(A,(x)B)$ for arbitrary pretypes $A$ and $B$ since universe types can be defined in $\mlsi$ only in the case in which $A$ and $B$ are (decodings of terms) in the superuniverse $\mathcal{S}$.}\]
\item the superuniverse is interpreted as $\mathcal{S}^{I}:=\{x|\,\mathsf{Set}(x)\}$ while its decodings are interpreted as $(\mathsf{T}(a))^{I}:=\{x|\,x\,\varepsilon\,a^I\}$;
\item the interpretation of terms $t(a,(x)b,c)$ representing the decoding in the superuniverse (which is itself still a code) of a code $c$ in the universe generated by the family $b$ on $a$, is simply defined as $c^I$, while the interpretation of decodings relative to the universe $\mathsf{U}(\mathsf{T}(a),(x)\mathsf{T}(b))$ is defined as 
  \[(\mathsf{T}(\mathsf{T}(a),(x)\mathsf{T}(b),c))^{I}:=\{x|\,x\,\varepsilon\, c^I\}\] 
\end{enumerate}
Hence we have the following theorem.
\begin{thm}\label{real2}
        Consistency of the theory $\mathbf{CZF}+\mathbf{REA}$  implies the consistency of the  theory $\mlsi$ extended with the formal Church thesis \ct.
        \end{thm}
\begin{cor}\label{main1}
        Consistency of the theory $\mathbf{CZF}+\mathbf{REA}$  implies the consistency of the  theory \mttcind\ extended with the formal Church thesis \ct\ and the axiom of choice
         $\ac$.
        \end{cor}
        \begin{proof}
        This is a consequence of Corollary~\ref{cindinsup} and Theorem~\ref{real2}.
        \end{proof}
       Moreover, we know from Theorem 4.6 in~\cite{mmr21} that $\CZF+\mathbf{REA}$ and $\mluif$ have the same proof-theoretic  strength. Since $\mluif$ is a subsystem of $\mlucf$ 
       which can be interpreted in $\mlsi$ as in Corollary~\ref{cindinsup}, from Theorem~\ref{real2} it follows that $\mlsi$ and $\CZF+\mathbf{REA}$ have the same proof-theoretical strength, as well as $\mluif$ and $\mlucf$.
\begin{cor}\label{pts} The following theories share the same proof-theoretic strength.
\begin{enumerate}
\item  $\CZF+\mathbf{REA}$
\item $\mluif$
\item $\mlucf$
\item $\mlsi$.
\end{enumerate}
\end{cor}

\begin{rem}\label{Wtypes}\em
It is worth noting that the above theorems still hold if we replace inductively generated formal topologies with the closure of \mtt-sets with  well founded trees in~\cite{PMTT}, also called $W$-types,
which we simply call $W$-sets when added to \mtt.
 Indeed we can prove that such a extension $\bf \mtt + \mbox{\bf $W$-sets}$ can be interpreted in the extension $\mathbb{MLS}_W$ of the first-order fragment of intensional Martin-L{\"o}f's type theory in~\cite{PMTT} with
 a superuniverse $\mathcal{S}$ \`a la Tarski closed under well founded trees, i.e. $W$-types,  besides the usual first-order
 type constructors and universe constructors  as in~\cite{suppal}.
 Then we can prove that $\mathbb{MLS}_W$ is consistent with $\ct$ and hence $\bf \mtt + \mbox{\bf $W$-sets}$  is consistent with $\ac+\ct$.
The proof that  $\mathbb{MLS}_W$ is consistent with $\ct$ can be obtained by building a realizability interpretation as that in this section which is in turn based on that of Section~\ref{sec-real}. This  is obtained by removing clause (7) in Section~\ref{sec-real}  and substituting it with clause (7'), after having introduced the abbreviations $\mathsf{w}(a,b):=\pp(11,\pp(a,b))$ and $\widetilde{\mathsf{sup}}(a,b):=\pp(12,\pp(a,b))$, as follows:
\begin{enumerate}
\item[(7')] Let $\beta\in \alpha$. Suppose that the following conditions are satisfied: 
\begin{enumerate}
\item $\mathsf{Set}_\beta(a)$, 
\item $\mathsf{Fam}_\beta(b,a)$.
\end{enumerate}
 Then $\mathsf{Set}_\alpha(\mathsf{w}(a,b))$.
\item[] For $\beta\in\alpha$ satisfying $\mathsf{Set}_\beta(a)$ and $\mathsf{Fam}_\beta(b,a)$, let $H_{\beta}$ be the smallest subset of $\mathbb{N}$ satisfying the following condition: if $c\,\varepsilon_{\beta}\,a$ and $\forall z\,\varepsilon_{\beta}\{b\}(c)\ (\{d\}(z)\in  H_\beta)$, then $\widetilde{\mathsf{sup}}(c,d)\in H_\beta$. 
 The existence of the set $H_\beta$ is guaranteed by the axiom $\mathbf{REA}$ and we define $h\,\varepsilon _\alpha\,\mathsf{w}(a,b)$ by $\exists \beta\in \alpha\,[\mathsf{Set}_\beta(a)\wedge \mathsf{Fam}_{\beta}(b,a) \,\wedge\, h\in H_{\beta}]$.
 
The interpretations of elimination terms $wrec$ (see p.99 pf~\cite{PMTT}) relative to W-types are obtained as codes $\mathbf{n}_{q}$ of recursive functions depending primitively recursively on the parameters in $q$ satisfying
    \[\{\mathbf{n}_q\}(\widetilde{\mathsf{sup}}(c,d))\simeq\{\Lambda x.\Lambda y.\Lambda z.q^I\}(c,d,\Lambda v.\{\mathbf{n}_q\}(\{d\}(v)))\]

\end{enumerate}

\end{rem}

\section{$\CZF$ with large set existence axioms and the superuniverse in type theory}
The axioms  $\mathbf{REA}$ and  $\mathbf{REA}_{\bigcup}$ of the previous section section can be viewed as constructive analogs of regular cardinals (see~\cite{rl03,czf2}).\footnote{$\mathbf{ZF}$ alone, though, cannot prove the existence of regular sets containing $\omega$. This follows from~\cite[Corollary 7.1]{rl03}.}
Further strengthenings of  the notion of regularity lead to weakly inaccessible, inaccessible and Mahlo sets (see~\cite{rgp98,cr02} and~\cite{czf2} chapter 18).\footnote{The terminology
varies between different papers. Inaccessible sets are called set-inaccessible in~\cite{rgp98} while weakly inaccessible sets are called inaccessible in~\cite{cr02,czf,czf2}. The terminology in the present paper is the same as used in~\cite{fefrat}. On the basis of $\mathbf{ZFC}$, though, the notions coincide.} 
In a constructive environment such as $\CZF$, however, the existence of such sets does not entail the enormous proof-theoretic strength they engender  in the classical context.
Indeed, the existence of weakly inaccessible sets does not add more strength to $\CZF+\mathbf{REA}$. In more detail, if one adds to $\CZF$ the axiom $\wINAC$ asserting that every set is contained in a weakly inaccessible set then $\CZF+\wINAC$ possesses  a recursive realizability interpretation in $\CZF+\mathbf{REA}$. This follows basically from~\cite[Theorem 4.7]{antirat}, when one changes the interpreting theory classical theory $\mathbf{KPi} $ therein to the theory $\CZF+\mathbf{REA}$ by deploying a constructivization of the techniques of~\cite{R93,RG94} (details will be provided after Definition~\ref{mlsr}).   Furthermore, there is a close connection between superuniverses in type theory (see~\cite{suppal,rgp98,ra00,suprat1,suprat2}) and the axiom $\wINAC$
in set theory.

\begin{defi}\label{letzt}
We say that a set is \emph{weakly inaccessible} if it is 
regular and a model of $\CZF$. A set is {\bf inaccessible} if it is regular and a model of $\CZF+\mathbf{REA}$. 
\end{defi}

The formalization  of the notion of  inaccessibility in Definition~\ref{letzt}
is somewhat  syntactic  in that it requires a
satisfaction predicate for formulae interpreted over a set. An alternative and more
`algebraic' characterization can be given as follows.
\begin{defi}Let $\Omega :=\{x:\;x\subseteq\{0\}\}$. $\Omega$ is the class of
truth values with $0$ representing falsity and $1=\{0\}$ representing truth.
Classically one has $\Omega=\{0,1\}$ but intuitionistically one cannot conclude
that those are the only truth values.

For $a\subseteq\Omega$ define 
\begin{eqnarray*} \bigwedge a &=& \{x\MRinn 1:\;(\forall u\MRinn a)x\MRinn
u\}.\end{eqnarray*}  
A class $B$ is \emph{$\bigwedge$-closed} if for all $a\MRinn B$, whenever $a\subseteq
\Omega$, then $\bigwedge a\in B$.

 For sets $a,b$ let $\MRhoch{a}{b}$ be the class of all 
functions with domain
$a$ and with range contained in $b$.
Let $\MRvoll{a}{b}$ be the class of all  sets $r\subseteq a\times b$ 
satisfying $\forall u\MRinn a\,\exists v\MRinn b\,\MRpaar{u,v}\MRinn r$.
A set $c$ is said to be \emph{full in} $\MRvoll{a}{b}$ if $c\subseteq \MRvoll{a}{b}$
  and \[\forall r\MRinn\MRvoll{a}{b}\,\exists s\MRinn c\, s\subseteq r.\]

The expression $\MRvoll{a}{b}$ should be read as the collection of
\emph{multi-valued functions} from  $a$ to  $b$.
\end{defi}
\begin{propC}[($\CZF$)] A set $I$ is weakly inaccessible if and only if the 
following are satisfied:
\begin{enumerate}
\item $I$ is a regular set,
\item $\omega\MRinn I$,
\item $(\forall a\MRinn I)\bigcup a\MRinn I$,
\item $I$ is $\bigwedge$-closed,
\item $(\forall a,b\in I)(\exists c\MRinn I)\,\bigl[\mbox{$c$ is full in $\MRvoll{a}{b}$}
\bigr]$.
\end{enumerate} 
\end{propC}
\begin{proof} See~\cite{ra00}, Proposition 3.4.
\end{proof}

\begin{rem}\label{rem} Note that $\MRCZF+\mathbf{REA}_{\bigcup}$ is a subtheory of  $\MRCZF+\wINACC$. \end{rem}

Viewed classically and in the presence of the axiom of choice,  weakly inaccessible sets give rise to strongly inaccessible cardinals, i.e., regular cardinals $\kappa>\omega$ such that 
$2^{\rho}<\kappa$ forall $\rho<\kappa$.

Let $V_{\alpha}$ denote the $\alpha$th level of the von Neumann hierarchy. 

\begin{propC}[($\MRZFC$)] If $I$ is a weakly inaccessible set then
$I=V_{\kappa}$ for some strongly inaccessible cardinal $\kappa$.\end{propC}
\begin{proof} This is a consequence of the proof of~\cite{rgp98}, Corollary 2.7.
\end{proof}

The next  result from~\cite{antirat} shows that the strength of $\wINACC$ is quite modest when
based on constructive set theory.

\begin{thm} $\MRCZF+\MRREA$ and $\MRCZF+\wINACC$ have the same proof-theoretic strength
as the subsystem of second order arithmetic with
${\MRboldsymbol\Delta}^1_2$-comprehension and bar induction.
\end{thm}
\begin{proof}~\cite{antirat} Theorem 4.7.\end{proof}
Indeed,~\cite{antirat} basically furnishes interpretations between  $\MRCZF+\MRREA$ and $\MRCZF+\wINACC$ and a version of type theory $\MRMLS$ with a superuniverse as we shall argue below.

\begin{defi}\label{mlsr}
 The type theory $\MRMLS$ has the
following ingredients:\footnote{Regarding the exact formalization of a
superuniverse $\MRmathbb S$ and the universe operator $\mathbf U$ see~\cite{ra00}.} 
\begin{itemize}
\item $\MRMLS$ demands closure
 under the usual
type constructors $\Pi$, $\Sigma$, $+$, 
$I$, ${\MRmathbb N}$, ${\MRmathbb N}_0$, ${\MRmathbb N}_1$ (but not the
$W$-type).
\item $\MRMLS$ has a superuniverse $\MRmathbb S$ which is closed under
$\Pi$, $\Sigma$, $+$, $I$, ${\MRmathbb N}$,
${\MRmathbb N}_0$, ${\MRmathbb N}_1$ and the $W$-type and the universe
operator $\mathbf U$.
\item $\MRMLS$  has a type $\mathbf V$ of iterative sets over $\MRmathbb S$ (see~\cite{aczel78}).
\item The universe operator $\mathbf U$ takes a type $A$ in $\MRmathbb S$ and a family
of types $B:A\rightarrow {\MRmathbb S}$ and produces a universe ${\mathbf U}(A,B)$ in
$\mathbb S$ which contains $A$ and $B(x)$ for all $x\in A$, and is closed under
$\Pi$, $\Sigma$, $+$, $I$, ${\MRmathbb N}$, ${\MRmathbb N}_0$, ${\MRmathbb N}_1$ (but not the
$W$-type).\end{itemize}
\end{defi}

The type $\mathbf V$ then enables one to perform an  interpretation of $\MRCZF+\wINACC+\mathbf{RDC}$ in $\MRMLS$
\`{a} la Aczel (see~\cite{aczel78,aczel82,aczel86}), using the techniques of~\cite{rgp98} section 4. The universe operator is crucial 
for modelling the weakly inaccessible sets in the manner of~\cite{rgp98} section 4. Note that $\mathbf{RDC}$ is modelled by this interpretation, too, as follows from~\cite{aczel82}.

Conversely, 
$\MRMLS$ has a recursive realizability interpretation in  $\MRCZF+\mathbf{REA}$. This uses a constructive modification of the techniques of 
~\cite{R93,RG94} sections 4 and 5. Although the latter articles give interpretations in versions of classical Kripke-Platek set theories ($\mathbf{KP}$ and $\mathbf{KPi}$) the interpretations
are constructively valid once one jettisons the unnecessary trichotomy of dividing ordinals into 0, successor and limit cases, as it was done in our previous paper~\cite{mmr21} in section 4.
The recursive realizability interpretation of the universe operator follows in the same way as in Definition 4.12,  Theorem 4.13 and Lemma 5.7 of~\cite{R93,RG94}.
Instead of using an admissible set as in~\cite{R93,RG94} Lemma 5.7 to show that the inductive definition of a universe gives rise to a set one employs a regular set $A$ containing all the parameters and uses recursion on the ordinals in the rank of $A$ to build up the set that models the universe.

So the upshot of the foregoing is the following result, mainly extracted from~\cite{antirat}.

\begin{thmC}[~\cite{antirat}]\label{Haupt1} The theories $\CZF+\mathbf{REA}$, $\CZF+\wINACC+\mathbf{RDC}$ and $\MRMLS$ are of the same strength. More precisely, they can be mutually interpreted in each other via the 
above interpretations. \end{thmC}

The type  $\mathbf V$ plays a crucial role in the interpretation of set theory in type theory  \`{a} la Aczel. However, it doesn't add any proof-theoretic strength.
This phenomenon was first observed in~\cite{R93,RG94} section 6. Let $\mathbb{MLS}$ be $\MRMLS$ without the type $\mathbf V$. Since the interpretation of the theory $\mathbf{IRA}$  of~\cite{R93,RG94} Definition 6.5 is also an interpretation in $\mathbb{MLS}$ we have the following result by~\cite{R93,RG94} Theorem 6.13.

\begin{thm}\label{Haupt2} The theories $\CZF+\mathbf{REA}$, $\CZF+\wINACC$, $\mathbb{MLS}$  and $\MRMLS$ are of the same proof-theoretic strength. 
\end{thm}
Finally, let us address the issue of adding Church's thesis, $\mathbf{CT}$, to various type theories. Of course, for this the $\xi$-rule has to be dropped  in the same way as in
~\cite{IMMS} and~\cite{mmr21} as explained in Definition~\ref{xixi}.  So, in future, if we write  $\mathbf{ML}+\mathbf{CT}$ where $\mathbf{ML}$ is a system of type theory, it is understood that this is type theory without the $\xi$-rule.

From Theorems~\ref{Haupt1} and~\ref{Haupt2} (or rather their proofs) it follows that the modified type theories with $\mathbf{CT}$ are consistent and, moreover, that they have the same strength as their cousins with the $\xi$-rule. The main reason for the type theories with $\mathbf{CT}$ being consistent is that we use recursive realizability interpretations for the type theories as in~\cite{mmr21}. On the other hand, the reason for the fact that the strength does not drop is that for the interpretation of set theory (or $\mathbf{IRA}$) in type theory  the $\xi$-rule doesn't matter at all. 

Therefore we conclude:
\begin{thm}\label{pts2} The following theories share the same proof-theoretic strength.
\begin{enumerate}
\item  $\CZF+\mathbf{REA}$
\item $\CZF+\wINACC+\mathbf{RDC}$ 
\item $\mathbb{MLS}$  
\item  $\MRMLS$ 
\item $\mathbb{MLS}+\mathbf{CT}$
\item  $\MRMLS+\mathbf{CT}$ 
\end{enumerate}
\end{thm}


As a consequence of Proposition~\ref{inmtt}, Theorem~\ref{real}, Remark~\ref{rem} and Theorem~\ref{Haupt2} we obtain again the following result.
\begin{cor}\label{main2}
        Consistency of the theory $\mathbf{CZF}+\mathbf{REA}$  implies the consistency of the  theory \mttcind\ extended with the formal Church thesis $\ct$ and axiom of choice $\ac$.
        \end{cor}

\begin{rem}\em
In \mttcind\ we added coinductive predicates  by inserting primitive proof-terms in an axiomatic way justified by the fact that such predicates become proof-irrelevant within
 \emttcind.  Indeed we are not aware of  explicit well-behaved rules for coinductive types to be added to Martin-L{\"o}f's type theory or to the Calculus of Constructions  (as mentioned in~\cite{failure}) and hence 
 to the intensional level of \mf.

On the other hand, current encodings of coinductive types like those in~\cite{coinMLTT, coinHott} are perfomed within versions of type theory validating extensional
equality of functions which can not be   validated in our Kleene realizability interpretation for  its inconsistency
with $ \ac+\ct$.  

What we can currently say  is that coinductive types in~\cite{coinMLTT}  are consistent  with Formal Church thesis $\ct$  over the extensional level $\emtt$ of \mf\  extended with well founded trees, called $W$-sets
in \mf-terminology.
Indeed we can encode coinductive types, or better  coinductive sets  in \mf-terminology, following~\cite{coinMLTT} within $\emtt+ \mbox{\bf $W$-sets}$  because this theory validates the assumptions
needed for the enconding in~\cite{coinMLTT}  due to the fact that  \emtt\ extends the first order version of  extensional Martin-Loef's type theory and validates  $0 \neq 1$   as shown in~\cite{m09}.
Moreover,  a two-level extension of \mf\  extended with $W$-sets  can  be provided as described in~\cite{MFwtypes}.
Then,   a Kleene realizability interpretation for $\mtt \ +\ \mbox{ \bf $W$-sets } + \ac +\ct$ can be built  as that for \mttind\  in~\cite{mmr21}, see remark~\ref{Wtypes}.
By composing such an  interpretation with the interpretation of \emtt\ in \mtt\  in~\cite{m09} extended to \mbox{ \bf $W$-sets } in~\cite{MFwtypes}, we obtain a realizability interpretation for  $\emtt + \ \mbox{ \bf W-sets } + \ct$, which
guarantees  that the encoded coinductive sets are consistent with $\ct$. Such an interpretation does not validate  $\ac$ since $\ac$  is constructively incompatible with \emtt\ as shown in~\cite{m09}.
\end{rem}

 \subsection*{Conclusions} We  have shown that  the intensional level  \mttcind\  of the extension \mfcind\ of the Minimalist Foundation \mf\  in~\cite{m09} is   consistent with $\ac+ \ct$ in two different ways.
In both ways we  show this by extending Kleene realizability interpretation of intuitionistic arithmetics in a constructive theory  whose consistency strength is that  of  {\bf CZF +REA}.  A  key benefit of the first way   is that  the intermediate theory $\mathbf{CZF}+\mathbf{REA}_{\bigcup}+\mathbf{RDC}$ also supports the  intended set-theoretic interpretation of the extensional level \emttcind\ of \mfcind.
 
%
 This work lets us conclude that the  extension \mfcind\  of   \mf\  with all the inductive and coinductive methods  developed in the field of Formal Topology
 constitutes a two-level foundation in the sense of~\cite{mtt}. 
 Moreover,   we confirmed  the expectation that the  addition of   coinductive topological definitions to \mttind\  to form \mttcind\  does not  increase its  consistency strength.

 Finally, all the theories used  to reach our goal, except  \mttcind\ and \mttind,  have shown  of the same proof-theoretic strength. 

  We leave it to future work to establish    the consistency strength  of \mttcind\ and   \mttind\
given that it is still an open problem to establish that of \mtt\ itself.

Another future goal would be to apply   the  realizability interpretations presented here  to build  predicative variants of Hyland's Effective Topos as in~\cite{misam21} but in a constructive meta-theory
such as  {\bf CZF +REA}.

\subsection*{Acknowledgments}The first author acknowledges very helpful discussions and suggestions with  U. Berger, F. Ciraulo, M. Contente, P. Martin-L{\"o}f, C. Sacerdoti Coen, G. Sambin and T. Streicher. The third author was supported by a grant from the John Templeton Foundation
(``A new dawn of intuitionism: mathematical and philosophical advances,'' ID 60842). 
                      
         \bibliographystyle{alphaurl}  		
\bibliography{bibliopspmich}
\end{document}